\documentclass[12pt,a4paper]{article}
\usepackage{mathrsfs,amsfonts}
\usepackage{amsthm}
\usepackage{amssymb}
\usepackage{latexsym}
\usepackage{amsmath}
\usepackage{mathrsfs}
\usepackage{cases}
\usepackage{latexsym,bm}
\usepackage{pict2e, color}
\usepackage{ifpdf}
\usepackage{graphicx}
\usepackage{psfrag}

\usepackage[varg]{txfonts}

\usepackage[
pdfauthor={JSY},
pdftitle={partial degree},
pdfstartview=XYZ,
bookmarks=true,
colorlinks=true,
linkcolor=blue,
urlcolor=blue,
citecolor=blue,
bookmarks=true,
linktocpage=true,
hyperindex=true
]{hyperref}

\usepackage{graphicx}
\usepackage{color}
\usepackage{indentfirst}

\title{Disjoint cycles covering specified vertices in bipartite graphs with partial degrees
\thanks{This work is supported by NNSF of China (No. 11901246 (to SJ), 12071260 (to JY)) and Research Fund of Jianghan University (No.1019/06310001) (to SJ).
}
}

\author{Suyun Jiang$^1$\thanks{Corresponding author. Emails: jiang.suyun@163.com (S. Jiang),  yanj@sdu.edu.cn (J. Yan) } \hspace{+12pt}
Jin Yan$^2$
\unskip\\[.5em]
{\small $^1$  Institute for Interdisciplinary Research, Jianghan University, Wuhan, Hubei 430056, China }\\
{\small $^2$
School of Mathematics, Shandong University, Jinan, Shandong 250100, China
}
}

\date{}

\newtheorem{thm}{Theorem}
\newtheorem{lem}{Lemma}[section]

\newtheorem{conjecture}{Conjecture}

\newtheorem{clm}{Claim}[section]

\theoremstyle{definition}
\newtheorem{definition}{Definition}
\theoremstyle{remark}

\numberwithin{equation}{section}

\begin{document}
\setlength{\baselineskip}{18pt}

\maketitle

\vspace{-24pt}

\begin{abstract}
Let $k$ be a positive integer. Let $G$ be a balanced bipartite graph of order $2n$ with bipartition $(X, Y)$, and $S$ a subset of $X$. Suppose that every pair of nonadjacent vertices $(x,y)$ with $x\in S, y\in Y$ satisfies $d(x)+d(y)\geq n+1$.
We show that if $|S|\geq 2k+2$, then $G$ contains $k$ disjoint cycles covering $S$ such that each of the $k$ cycles contains at least two vertices of $S$. Here, both the degree condition and the lower bound of $|S|$ are best possible. And we also show that if $|S|=2k+1$, then $G$ contains $k$ disjoint cycles such that each of the $k$ cycles contains at least two vertices of $S$.


\end{abstract}

\section{Introduction}
\label{Introduction}
All graphs in this paper are finite and simple. For terminology and notation not defined, we refer the readers to \cite{Diestel}. Let $G$ be a graph with vertex set $V(G)$ and edge set $E(G)$. If $G$ is a bipartite graph with bipartition $(X, Y)$, then we denote it by $G[X, Y]$, and say $G[X, Y]$ is balanced if $|X| = |Y|$.
For a vertex $v$ of $G$, we denote by $d_G(v)$ and $N_G(v)$ the  \textit{degree} and the \textit{neighborhood} of $v$ in $G$, respectively. 
For a subset $S$ of $V(G)$, we define $\delta(S)=\min\{d_G(v): v\in S\}$ and for $G=G[X,Y]$ we define
$$\sigma_{1,1}(S) = \min\{d_G(x) + d_G(y) : x\in X\cap S, y\in Y \mbox{\ or\ } x\in X, y\in Y\cap S, xy \notin E(G)\}$$
 if $G[X\cap S,Y]$ or $G[X,Y\cap S]$ is not a complete bipartite graph; otherwise, define $\sigma_{1,1}(S)=\infty$.
In particular, we denote $\delta(V(G))$ by $\delta(G)$ and denote $\sigma_{1,1}(S)$ by $\sigma_{1,1}(G)$ if $S\supseteq X$ or $S\supseteq Y$.
If the graph $G$ is clear from the context, we often omit the graph parameter $G$ in the graph invariant. 


For $S\subseteq V(G)$, we call $G$ \textit{$S$-cyclable} if $G$ has a cycle covering $S$. 
Bollob\'{a}s and Brightwell \cite{Bollobas} and Shi \cite{Shi} considered the sufficient partial degree condition $\delta(S)$ for graphs being $S$-cyclable and gave the following result. 

\begin{thm}[Bollob\'{a}s and Brightwell \cite{Bollobas}, Shi \cite{Shi}]\label{Shi}
Let $G$ be a 2-connected graph of order $n$ and $S\subseteq V(G)$. If $\delta(S)\geq \frac{n}{2}$, then $G$ is $S$-cyclable.
\end{thm}

%


%

Abderrezzak, Flandrin and Amar \cite{Abderrezzak} considered the bipartite version of Theorem \ref{Shi}.

\begin{thm}[Abderrezzak, Flandrin and Amar \cite{Abderrezzak}]\label{Abderrezzak}
Let $G[X,Y]$ be a 2-connected balanced bipartite graph of order $2n$ and $S\subseteq X$. If $\sigma_{1,1}(S)\geq n+1$, then $G$ is $S$-cyclable.
\end{thm}

When $S=V(G)$ or $S=X$ if $G=G[X,Y]$ and $G[X,Y]$ is balanced, $G$ is $S$-cyclable iff $G$ is hamiltonian. Thus Theorem \ref{Shi} implies Dirac’s Theorem \cite{Dirac} and Theorem \ref{Abderrezzak} implies the result of Moon and Moser \cite{Moon} which states that if $G$ is a balanced bipartite graph of order $2n$ with $\sigma_{1,1}(G)\geq n+1$ then $G$ is hamiltonian.

%
%

It is natural to consider the sufficient partial degree conditions for disjoint cycles or disjoint cycles covering $S$. For convenience, we first give the following definition.

\begin{definition}
Let $k$ be a positive integer. Let $S\subseteq V(G)$, and $S\subseteq X$ if $G=G[X,Y]$.

(1) A cycle $C$ of $G$ is \textit{feasible} if $|V(C)\cap S|\ge 3$, or $|V(C)\cap S|\ge 2$ if $G=G[X,Y]$;

(2) A graph $G$ is \textit{$S$-$k$-feasible} if $G$ contains $k$ disjoint feasible cycles;

(3) A graph $G$ is \textit{$S$-$k$-cyclable} if $G$ contains $k$ disjoint feasible cycles covering $S$.
\end{definition}

Wang \cite{Wang Partial} considered the condition $\delta(S)$ forcing a graph to be $S$-$k$-feasible and showed that if $G$ is a graph of order $n$ and $S$ is a subset of $V(G)$ with $\delta(S)\geq \frac{2n}{3}$ and $|S|\ge 3k$, then $G$ is $S$-$k$-feasible. Motivated by Wang’s result, we consider the condition $\sigma_{1,1}(S)$ for bipartite graphs to be $S$-$k$-feasible or $S$-$k$-cyclable and give the following two theorems.



%

%


\begin{thm}\label{theorem 2}
Let $G[X,Y]$ be a balanced bipartite graph of order $2n$ and $S\subseteq X$. If $\sigma_{1,1}(S)\geq n+1$ and $|S|\geq 2k+1$, then $G$ is $S$-$k$-feasible. Moreover, $G$ contains $k$ disjoint feasible cycles such that either these cycles cover $V(G)$ or each cycle contains exactly two vertices of $S$.
\end{thm}

 
\begin{thm}\label{theorem}
Let $G[X,Y]$ be a balanced bipartite graph of order $2n$ and $S\subseteq X$. If $\sigma_{1,1}(S)\geq n+1$ and $|S|\geq 2k+2$, then $G$ is $S$-$k$-cyclable.
\end{thm}

\noindent\textbf{Remark.} (1) The lower bound of $|S|$ in Theorems  \ref{theorem 2} and \ref{theorem} and the condition $\sigma_{1,1}(S)$ in Theorem \ref{theorem} all are best possible, and $\sigma_{1,1}(S)\ge n$ is necessary for Theorem \ref{theorem 2} (see Section \ref{concluding remarks}).

(2) When $|S|\ge 4$, Theorem \ref{theorem} implies 
Theorem \ref{Abderrezzak} by letting $k=1$. 

(3) In \cite{Jiang}, we showed that a balanced bipartite graph $G[X,Y]$ of order $2n$ is $S$-$k$-feasible if $S$ is a subset of $X$ with $\sigma_{1,1}(S)\geq n+k$ and $|S|\geq 2k$. Thus Theorem \ref{theorem 2} implies the result of \cite{Jiang} when $|S|\ge 2k+1$.

(4) Note that if $G[X,Y]$ is a balanced bipartite graph and is $X$-$k$-cyclable then $G[X,Y]$ has a 2-factor with $k$ cycles. Thus Theorem \ref{theorem} implies the result of Chiba and Yamashita \cite{Chiba2017} which states that if $G$ is a balanced bipartite graph of order $2n\ge 2(12k+2)$ and $\sigma_{1,1}(G)\ge n+1$ then $G$ has a 2-factor with $k$ cycles.



The rest of the paper is organized as follows: we present notations and some useful lemmas in Section \ref{pre}, and give the proof of Theorem \ref{theorem 2} and Theorem \ref{theorem} in Sections \ref{pftheorem1} and \ref{pftheorem2}, respectively, and we will give four examples to illustrate Remark (1) and some problems in Section \ref{concluding remarks}.

%

\section{Preliminaries}
\label{pre}

\subsection{Basic terminology and notation}

Let $G[X,Y]$ be a balanced bipartite graph. 
Let $U$ be a subset of $V(G)$ and $H$ a subgraph of $G$.
We denote by $G[U]$ the subgraph of $G$ induced by $U$, and we define $G-H=G[V(G)\setminus V(H)]$ and $U_{H}=U\cap V(H)$. 
And we call a matching $M$ of $H$ a \textit{$U$-matching} if $X_M\subseteq U$, and a $U$-matching of size $k$ is called a \textit{$k$-$U$-matching}.
Let $G_1$ and $G_2$ be subgraphs of $G$. We denote by $G_1\cup G_2$ the union of $G_1$ and $G_2$, i.e, the subgraph with vertex set $V(G_1)\cup V(G_2)$ and edge set $E(G_1)\cup E(G_2)$. And denote by $E(G_1,G_2)$ the set of edges of $G$ with one end in $V(G_1)$ and the other end in $V(G_2)$, and by $e(G_1,G_2)$ their number. Clearly, $e(G_1,G_2)=\sum_{v\in V(G_1)}d_{G_2}(v)$.

If we write a path $P=v_1v_2\cdots v_m$ or cycle $C=v_1v_2\cdots v_mv_1$, we assume that an orientation of $P$ or $C$ is given such that $v_{i+1}$ is the successor of $v_i$ on $P$ or $C$, and $v_1$ is the successor of $v_m$ on $C$, where $1\le i\le m-1$. Moreover, we use $v_i^+$ and $v_i^-$ to denote the successor and predecessor of $v_i$ on $P$ or $C$ (provided that these vertices exist), respectively. And for any subset $A$ of $V(P)$ or $V(C)$, we define $A^+=\{v^+:v\in A\}$. Furthermore, we use $C[v_i,v_j]$ to represent the path of $C$ from $v_i$ to $v_j$ along the orientation of $C$. Similarly, we define $P[v_i,v_j]$.


A set of disjoint cycles $\{C_1,\ldots, C_k\}$ is called a \textit{minimal system} of $G$ if $C_1,\ldots, C_k$ are feasible and $G$ does not contain disjoint feasible cycles $C_1^\prime, \ldots, C_k^\prime$ such that $\sum_{i=1}^k|C_i^\prime|<\sum_{i=1}^k|C_i|$.
And we call a minimal system $\{C_1,\ldots, C_k\}$ of $G$ a \textit{minimal $S$ system} of $G$ if $G$ does not contain disjoint feasible cycles $C_1^\prime, \ldots, C_k^\prime$ such that $\sum_{i=1}^k|C_i^\prime|=\sum_{i=1}^k|C_i|$ and $\sum_{i=1}^k|S_{C_i^\prime}|<\sum_{i=1}^k|S_{C_i}|$. For simplicity, we use $C$ instead of $\{C\}$.



\subsection{Lemmas}

In the following, $G[X,Y]$ is a balanced bipartite graph of order $2n$ and $S$ is a subset of $X$. Recall that a cycle $C$ is feasible if $|S_C|\geq 2$. First we study the structure of $G[V(C\cup H)]$ and give Lemmas \ref{lemma 1}-\ref{lemma 4}, where $C$ is a cycle of $G$, $H$ is a subgraph in $G-C$ consisting of distinct vertices, or disjoint edges, or a path.


\begin{lem}\label{lemma 1}
Let $C=x_1y_1\cdots x_my_mx_1$ be a cycle of $G$ with $x_1\in X$ and $|C|=2m\ge 6$, and let $x$ and $y$ be two vertices of $G-C$ with $x\in S$ and $y\in Y$. Suppose that $C$ is a minimal system of $G$. Then the following four statements hold.
\begin{itemize}
\setlength{\itemsep}{0pt}
\setlength{\parsep}{0pt}
\setlength{\parskip}{0pt}
  \item [\emph{(1)}]\emph{\cite{Jiang}} $d_C(x)\leq 2$ and $d_C(y)\leq 4$. Moreover, $e(\{x,y\}, C)\le 4$ if $xy\in E$.
  \item [\emph{(2)}] $d_C(y)\leq \frac{1}{2}|C|-1$. In particular, if $d_C(y)= \frac{1}{2}|C|-1$ and $X_C\setminus N_C(y)=\{x_j\}$, then $x_j\in S$, $|S_C|=2$, $|C|\le 8$ and $S_C=\{x_j, x_{j+2}\}$ if $|C|=8$.
  \item [\emph{(3)}] If $e(\{x,y\}, C)\geq \frac{1}{2}|C|+1$, then $xy\notin E$ and $G[V(C)\cup \{x,y\}]$ contains a cycle $C^\prime$ and an edge $xy_i$ such that $xy_i$ is disjoint from $C^\prime$, $|C^\prime|=|C|$ and $|S_{C^\prime}|=|S_C|$.
  \item [\emph{(4)}] $e(\{x,y\}, C)\le \frac{1}{2}|C|$ if $xy\in E$.
\end{itemize}
\end{lem}
\begin{proof}
Let $A=\{x_i\in X_C: |S_{C-x_i}|\geq 2\}$ and $B=\{x_i\in X_C: |S_{C-x_i}|\leq 1\}$. Clearly, $|A|+|B|=m$. Since $m\ge 3$ and $|S_C|\ge 2$, we have $|A|\ne 0$, $B\subseteq S_C$ and if $|B|\ne 0$ then $|S_C|=2$. For each vertex $x_i\in A$, we have $d_{\{x_{i-1}, x_{i+1}\}}(y)\le 1$ for otherwise  $yC[x_{i+1},x_{i-1}]y$ is a feasible cycle of order $|C|-2$, so $d_C(y)\leq \frac{1}{2}|C|-\left\lceil \frac{|A|}{2}\right\rceil\le \frac{1}{2}|C|-1$. If $d_C(y)= \frac{1}{2}|C|-1$ and $X_C\setminus N_C(y)=\{x_j\}$, then $A\subseteq \{x_{j-1}, x_{j+1}\}$, thus $x_j\in B\subseteq S_C$, $|S_C|=2$ and $|C|=2m\le 2(2+2)=8$. Clearly, if $|C|=8$, then $A=\{x_{j-1}, x_{j+1}\}$ and $S_C=B=\{x_j, x_{j+2}\}$.
Thus (2) holds.

Now we only need to prove (3). Suppose that $e(\{x,y\}, C)\geq \frac{1}{2}|C|+1$. By (1) and (2), the following hold: $d_C(x)=2$, $d_C(y)=\frac{1}{2}|C|-1$, $|S_C|=2$, $|C|\le 8$ and $yx_j\notin E$ for some $x_j\in S_C$, furthermore, $S_C=\{x_j, x_{j+2}\}$ if $|C|=8$. We may assume that $j=1$ and $S_C=\{x_1,x_3\}$.
Then $d_{\{y_1, y_m\}}(x)\le 1$ for otherwise $xy_mx_1y_1x$ is a feasible quadrilateral. Note that $d_C(x)=2$, we have $d_{Y_C-\{y_1, y_m\}}(x)\ge 1$, assume $xy_q\in E$, where $1<q<m\le 4$. It follows that $G[V(C)\cup \{x,y\}]$ contains a feasible cycle $yC[x_{q+1},x_q]y$ and an edge $xy_q$. Moreover, $xy\notin E$ for otherwise $xyx_3y_qx$ is a feasible quadrilateral. Thus (3) holds.
\end{proof}

\begin{lem}\label{lemma 2}\emph{\cite{Wang bipartite}}
Let $C$ be a quadrilateral of $G$ and $H$ a subgraph of $G-C$. Suppose that $e(H,C)\ge |H|+1$. Then the following four statements hold.
\begin{itemize}
\setlength{\itemsep}{0pt}
\setlength{\parsep}{0pt}
\setlength{\parskip}{0pt}
  \item [\emph{(1)}] If $|X_H|=|Y_H|=1$, then $G[V(C\cup H)]$ contains a quadrilateral $C^\prime$ and an edge $e$ such that $e$ is disjoint from $C^\prime$.
  \item [\emph{(2)}] If $|H|=4$ and $H$ contains two disjoint edges, then $G[V(C\cup H)]$ contains a quadrilateral $C^\prime$ and a path $P^\prime$ of order 4 such that $P^\prime$ is disjoint from $C^\prime$.
  \item [\emph{(3)}] If $|H|=4$ and $H$ is a path such that $e(H,C)\ge 6$, then either $G[V(C\cup H)]$ contains two disjoint quadrilaterals, or $H$ has an endvertex, say $z$, such that $d_C(z)=0$.
  \item [\emph{(4)}] If $|H|\ge 6$ and $H$ is a path, then $G[V(C\cup H)]$ contains two disjoint cycles.
\end{itemize}
\end{lem}

%
%
%

\begin{lem}\label{lemma 3}
Let $C$ be a feasible quadrilateral of $G$. Let $P=u_1u_2\cdots u_p$ be a path of $G-C$ with $|S_P|\ge 3$, $u_1\in S$, and $u_p\notin S$ if $|S_P|=3$.
Suppose that $e(U,C)\ge |U|+1$, where 
$$ U=
\begin{cases}
S_P\cup S_P^+,   \quad\quad     \mbox{if}~u_p\notin S;\\
S_P\cup S_{P-u_p}^+,  \quad    \mbox{if}~u_p\in S.
\end{cases} $$
Then $G[V(C\cup P)]$ is $S$-2-feasible.
\end{lem}

\begin{proof}
 Let $|S_P|=s$ and $S_P=\{u_{i_1}, \ldots, u_{i_s}\}$ with $1=i_1<\cdots <i_s\le p$. Let $G^\prime=G\cup \{u_{i_1}^+u_{i_2}, \ldots, u_{i_{s-1}}^+u_{i_s}\}$, and let $P^\prime=u_{i_1}u_{i_1}^+u_{i_2}\cdots u_{i_{s-1}}^+u_{i_s}u_{i_s}^+$ if $u_p\notin S$ and $P^\prime=u_{i_1}u_{i_1}^+u_{i_2}\cdots u_{i_{s-1}}^+u_{i_s}$ if $u_p\in S$. Then $|P^\prime|\ge 6$ and $e(P^\prime,C)\geq  |P^\prime|+1$. By Lemma \ref{lemma 2} (4), $G[V(C\cup P^\prime)]$ contains two disjoint cycles $C_1$ and $C_2$, which implies that we readily obtain two disjoint feasible cycles of $G[V(C\cup P)]$ from $C_1$ and $C_2$ by replacing the edges $u_{i_1}^+u_{i_2}, \ldots, u_{i_{s-1}}^+u_{i_s}$ with the paths $P[u_{i_1}^+, u_{i_2}]$, \ldots, $P[u_{i_{s-1}}^+, u_{i_s}]$, respectively. Thus $G[V(C\cup P)]$ is $S$-2-feasible.
\end{proof}

\begin{lem}\label{lemma 4}
Let $C$ be a cycle of $G$ and $P=P[x, y]$ a path of even order in $G-C$.

\emph{(1) \cite{Wang lem}}  If $e(\{x,y\},C)\ge \frac{|C|}{2}+1$, then $G[V(C\cup P)]$ is hamiltonian.

\emph{(2) \cite{Bondy1976}} If $|P|\ge 4$ and $e(\{x, y\}, P)\ge \frac{|P|}{2}+1$, then $G[V(P)]$ is hamiltonian.
\end{lem}


Now we study the structure of $G[V(P)]$, where $P$ is a path of $G$ with some specific properties, and the structure of $G$ if $G$ contains a minimal system or a minimal $S$ system, and give Lemmas \ref{lemma 5}-\ref{lemma 3.2}.

\begin{lem}\label{lemma 5}
Let $P=x_1y_1\cdots x_py_p$ be a path of even order with $x_1, x_p\in S$ and $p\ne 1$. Suppose that $G[V(P)]$ does not contain a feasible cycle or a shorter path $P^\prime$ of even order such that $S_{P^\prime}=S_P$. 
Then $d_P(x_1)=1$, $d_P(x_i)=2$ for each $x_i\in S_P\setminus \{x_1\}$ and $d_P(y_p)\le 2$.
\end{lem}
\begin{proof}
First suppose that $d_P(x_1)>1$ or $d_P(x_i)>2$ for some $x_i\in S_P\setminus \{x_1\}$. Then there exists an edge $x_iy_j\in E(G[V(P)]-P)$, where $x_i\in S_P$, $y_j\in Y_P$ and $j\notin \{i-1, i\}$. If $j<i-1$, then $G[V(P)]$ contains a feasible cycle $P[y_j, x_i]y_j$ if $|S_{P[y_j, x_i]}|\ge 2$ and $G[V(P)]$ contains a shorter path $P[x_1, y_j]P[x_i, y_p]$ if $|S_{P[y_j, x_i]}|=1$, a contradiction. Similarly, we can get a contradiction if $j> i$. 



Now suppose that $d_P(y_p)\ge 3$. Let $x_i, x_j\in N_P(y_p)$ with $i<j<p$. Then $G[V(P)]$ contains a feasible cycle $P[x_i, y_p]x_i$ if $|S_{P[x_i, y_p]}|\ge 2$, or a shorter path $P[x_1, x_i]y_px_py_{p-1}$ if $|S_{P[x_i, y_p]}|=1$, a contradiction. 
\end{proof}

By the similar proof of Lemma \ref{lemma 5}, if $C$ is a minimal system of $G$ and $|S_C|\ge 3$, then $d_{G[V(C)]}(u)=2$ for each $u\in S_{C}\cup S_C^+$.

Now we use Lemma \ref{lemma 6} to prove Lemma \ref{lemma 3.1}. Actually, Lemma \ref{lemma 3.1} can be deduced directly from the proof of Claim 3.1 in \cite{Jiang}, here we give the proof for completeness.

\begin{lem}\label{lemma 6}\emph{\cite{Jiang}}
Let $t$ and $s$ be two integers with $t\geq s\geq 2$ and $t\geq 3$. Let $\{C_1, C_2\}$ be a minimal system of $G$ such that $|S_{C_1}|=t$ and $|S_{C_2}|=s$. Then $e(S_{C_1}\cup S_{C_1}^+, C_2)\leq \frac{1}{2}t|C_2|$.
\end{lem}

\begin{lem}\label{lemma 3.1}
Let $k$ be a positive integer, and let $\{C_1, \ldots , C_k\}$ be a minimal $S$ system of $G$.
If $\sigma_{1,1}(S)\ge n+1$ and $|G-\cup_{i=1}^k C_i|\ne 0$, then $|S_{C_i}|=2$ for each $i\in \{1,\ldots,k\}$.
\end{lem}
\begin{proof}
On the contrary, suppose that Lemma \ref{lemma 3.1} fails. We may assume that $t=|S_{C_1}|=\max\{|S_{C_i}|: i\in\{1,\ldots, k\}\}$. Then $t\geq 3$. Let $H=G-\cup_{i=1}^kC_i$, $C_1=u_1v_1\cdots u_lv_lu_1$ and $S_{C_1}=\{u_{i_1}, \ldots ,u_{i_{t}}\}$, where $l=\frac{|C_1|}{2}$ and $1\le i_1<\cdots <i_{t}\le l$.
First we claim that
\begin{equation}\label{eqlemma3.1}
\mbox{\ for each\ } p, q \mbox{\ with\ } 1\le p<q\le t, |N_H(u_{i_p})\cap N_H(u_{i_q})|=|N_H(v_{i_p})\cap N_H(v_{i_q})|= 0.
\end{equation}
To prove (\ref{eqlemma3.1}), suppose that there exist integers $p, q$ and a vertex, say $a$, such that $a\in N_H(u_{i_p})\cap N_H(u_{i_q})$ or $a\in N_H(v_{i_p})\cap N_H(v_{i_q})$. Note that $|S_{C_1}|\ge 3$, we may assume $|S_{C_1[u_{i_q}^+, u_{i_p}^-]}|\ge 1$. Then $G[V(C_1\cup H)]$ contains a feasible cycle $C$ such that $|C|<|C_1|$ or $|S_C|<|S_{C_1}|$ if $|C|=|C_1|$, where 
$$ C=
\begin{cases}
aC_1[u_{i_p}, u_{i_q}]a,   \quad  \mbox{if}~a\in Y;\\
aC_1[v_{i_q}, v_{i_p}]a,  \quad \;   \mbox{if}~a\in X\setminus S;\\
aC_1[v_{i_p}, v_{i_q}]a,  \quad\;    \mbox{if}~a\in S,
\end{cases} $$
this contradicts $C_1$ is a minimal $S$ system of $G-\cup_{i=2}^kC_k$.

Let $L=S_{C_1}\cup S_{C_1}^+$. Then by (\ref{eqlemma3.1}), $e(L, H)\le |H|$. Since $t\geq 3$ and $\{C_1, \ldots , C_k\}$ is a minimal system of $G$, $d_{G[V(C_1)]}(u)=2$ for each $u\in L$ and $u_{i_t}v_{i_{1}}, u_{i_{j}}v_{i_{j+1}}\notin E$ for $1\leq j\leq t-1$, and by Lemma \ref{lemma 6} we have $e(L, \cup_{i=2}^kC_i)\le \frac{t}{2}\sum_{i=2}^k |C_i|$. It follows that 
\begin{equation*}
\mbox{\ $t(n+1)\le e(L, G)\le |H|+\frac{t}{2}\sum_{i=2}^k |C_i|+4t=t(n+1)-\frac{|H|}{2}(t-2)-t(\frac{|C_1|}{2}-3)< t(n+1)$, \ } 
\end{equation*}
where the last inequality holds because $|H|\ne 0$ and $t\ge 3$, a contradiction.
\end{proof}

\begin{lem}\label{lemma 3.2}
Let $k$ be a nonnegative integer, and let $\{C_1, \ldots , C_k\}$ be a minimal system of $G$. Then the following hold:

\emph{(1)} If $\sigma_{1,1}(S)\ge n$, then $G$ contains a minimal system $\{C_1^\prime, \ldots , C_k^\prime\}$ such that $|C_i^\prime|=|C_i|$, $|S_{C_i^\prime}|=|S_{C_i}|$ for each $i\in \{1,\ldots,k\}$, and $H^\prime$ has a $|S_{H^\prime}|$-$S$-matching, where $H^\prime=G-\cup _{i=1}^k C_i^\prime$.

\emph{(2)} If $\sigma_{1,1}(S)\ge n+1$ and $G-\cup _{i=1}^k C_i$ contains a path $Q=y_1y_2y_3$ with $y_2\in S$, then $G$ is $S$-$(k+1)$-feasible, or $G-Q$ contains a minimal system $\{C_1^\prime, \ldots , C_k^\prime\}$ such that $|C_i^\prime|=|C_i|$, $|S_{C_i^\prime}|=|S_{C_i}|$ for each $i\in \{1,\ldots,k\}$, and $H^{\prime}$ has a $l$-$S$-matching, where  $H^{\prime}=G-Q-\cup _{i=1}^k C_i^{\prime}$ and $l=\min\{|Y_{H^{\prime}}|, |S_{H^{\prime}}|\}$.
\end{lem}
\begin{proof}
Let $H_1=G-\cup _{i=1}^k C_i$, $H_2=H_1-Q$ (if $Q$ exists) and $M_j$ be a $S$-matching of $H_j$, where $j\in \{1, 2\}$. We choose $C_1,\ldots, C_{k}, M_j$ from $G$ or $G-Q$ (if $Q$ exists) such that $\{C_1, \ldots , C_k\}$ is a minimal system and subject to this, $|M_j|$ is maximum for each $j\in \{1, 2\}$. Suppose that $M_1$ does not satisfy (1) or $M_2$ does not satisfy (2). Then for some $j\in \{1, 2\}$ there exist two vertices $x\in S_{H_j-M_j}$ and $y\in Y_{H_j-M_j}$ such that $xy\notin E$.

According to the choice of $C_1, \ldots, C_{k}$ and $M_j$, we have $e(\{x,y\}, M_j)\le \frac{|V(M_j)|}{2}$, furthermore, Lemma \ref{lemma 1} (3) and Lemma \ref{lemma 2} (1) imply that $e(\{x,y\}, \cup _{i=1}^k C_i)\le \sum_{i=1}^k \frac{|C_i|}{2}$.
 If $j=1$, we have $d(x)+d(y)\le \sum_{i=1}^{k}\frac{|C_i|}{2}+\frac{|V(M_1)|}{2}+0+(\frac{|H_1-M_1|}{2}-1)= n-1<n$, a contradiction. Thus (1) holds. Now suppose $j=2$. If $d_Q(x)=2$, then $G$ contains $k+1$ disjoint feasible cycles $C_1, \ldots , C_k$ and $xy_1y_2y_3x$; if $d_Q(x)\le 1$, then $d(x)+d(y)\le \sum_{i=1}^{k}\frac{|C_i|}{2}+\frac{|V(M_2)|}{2}+1+(\frac{|H_1-M_2|}{2}-1)= n<n+1$, a contradiction. Thus (2) holds.
\end{proof}

%
%
%

\section{Proof of Theorem \ref{theorem 2}}
\label{pftheorem1}

Let $G[X,Y]$ be a balanced bipartite graph of order $2n$, and $S$ a subset of $X$ such that $|S|\geq 2k+1$ and $\sigma_{1,1}(S)\ge n+1$, where $k$ is a positive integer and $\sigma_{1,1}(S)=\min\{d(x)+d(y) : x\in S, y\in Y, xy\notin E\}$. If $G$ is $S$-$k$-feasible, then Theorem \ref{theorem 2} holds by Lemma \ref{lemma 3.1}. So suppose that $G$ is not $S$-$k$-feasible and $G$ is edge-maximal, that is, $G+xy$ is $S$-$k$-feasible for each pair of nonadjacent vertices $x\in X$ and $y\in Y$. Clearly, $G$ is $S$-$(k-1)$-feasible. In the folllowing, our aim is to get a contradiction by showing that $G$ is $S$-$k$-feasible. For this purpose, we describe two good structures in the below two claims, which allow us to transform $s$ disjoint feasible cycles to $s+1$ disjoint feasible cycles. 

%
%
%


\begin{clm}\label{claim xin1}
Let $\{D_1, \ldots, D_s\}$ be a minimal system of $G$. Suppose that $G-\cup_{i=1}^{s}D_i$ contains a path $P$ such that $|S_P|=3$ and $|P|$ is even. Then $G$ is $S$-$(s+1)$-feasible.
\end{clm}
\begin{proof}
Let $H=G-\cup_{i=1}^{s}D_i$ and $P=x_1y_1\cdots x_py_p$ with $S_P=\{x_1, x_i, x_p\}$, where $1<i<p$. Choose $P$ so that $|P|$ is minimum. Let $L=S_P\cup S_P^+$, where $S_P^+=\{y_1, y_i, y_p\}$.
If $\{x_1y_i, x_iy_p, x_py_1, x_1y_p\}\cap E\ne \emptyset$, then we are done. Now suppose $x_1y_i, x_iy_p, x_py_1, x_1y_p\notin E$, and so $e(L,G)\ge 3n+3$. It follows that one of the following inequalities holds: 
\begin{equation*}
\mbox{\ (1) $e(L, \cup_{i=1}^{s}D_i)\ge \sum_{i=1}^{s}\frac{3|D_i|}{2}+1$; $\ $ (2) $e(L, P)\ge \frac{3|P|}{2}+2$; $\ $ (3) $e(L, H-P)\ge \frac{3|H-P|}{2}+1$.\ }
\end{equation*}
If (1) holds, then by Lemma \ref{lemma 1} (4) and Lemma \ref{lemma 3} there exists a feasible quadrilateral, say $D_1$, such that $G[V(D_1\cup P)]$ is $S$-2-feasible, thus $G$ is $S$-$(s+1)$-feasible. 
Note that $|P|$ is minimum and $x_1y_i, x_iy_p, x_py_1, x_1y_p\notin E$. 
If (2) holds, then by Lemma \ref{lemma 5} we have $G[V(P)]$ is $S$-1-feasible or $e(L, P)\le 1+2+2+(\frac{3|P|}{2}-4)<\frac{3|P|}{2}+2$, thus $G$ is $S$-$(s+1)$-feasible.
If (3) holds, then $e(S_P, H-P)\ge \frac{|H-P|}{2}+1$ or $e(\{y_1, y_p\}, H-P)\ge \frac{|H-P|}{2}+1$, it follows that $G[V(P)]$ is $S$-1-feasible and so $G$ is $S$-$(s+1)$-feasible.
\end{proof}

\begin{clm}\label{claim xin2}
Let $\{D_1, \ldots, D_s\}$ be a minimal system of $G$. Suppose that $G-\cup_{i=1}^{s}D_i$ contains two disjoint paths $P=x_1\cdots x_{2p}$ and $Q=y_1y_2y_3$ such that $S_{P\cup Q}=\{x_1, x_{2p-1},y_2\}$. Then $G$ is $S$-$(s+1)$-feasible.
\end{clm}
\begin{proof}
Let $H=G-\cup_{i=1}^{s}D_i$.
On the contrary, suppose that $G$ is not $S$-$(s+1)$-feasible. Thus by Claim \ref{claim xin1}, we have for each $i\in \{1, \ldots, s\}$,
\begin{equation}\label{eqpath3}
 \begin{split}
& \mbox{\ $G[V(D_i\cup H)]$  is not $S$-2-feasible and it does not contain a cycle $D_i^\prime$ and a path $P_0$\ }\\
& \mbox{\ such that $|D_i^\prime|=|D_i|$, $|S_{D_i^\prime}|=|S_{D_i}|$, $|S_{P_0}|=3$ and $|P_0|$ is even.\ } 
  \end{split}
\end{equation}
Choose $P$ and $Q$ such that $|P|$ is minimum and $x_2y_2\in E$ if possible. Let $H^\prime=H-P-Q$ and $L=S_{P\cup Q}\cup \{x_{2p}, y_1, y_3\}$. 
By (\ref{eqpath3}) and the choice of $P$, we have $e(\{x_1, x_{2p-1}, x_{2p}\}, Q)=0$, $e(S_{P\cup Q}, H^\prime)\le |Y_{H^\prime}|=\frac{|H^\prime|-1}{2}$, $e(\{x_{2p}, y_1\}, H^\prime)\le |X_{H^\prime}|=\frac{|H^\prime|+1}{2}$, and $e(\{x_1, x_{2p-1}, x_{2p}\}, P)\le 1+2+2=5$ by Lemma \ref{lemma 5}.
It follows that $e(L,G)\ge 3n+3$ and $e(L, H)=e(L, H^\prime\cup Q\cup P)\le (\frac{|H^\prime|-1}{2}+2(\frac{|H^\prime|+1}{2}))+(0+4)+(5+(\frac{3|P|}{2}-5))=\frac{3|H|}{2}$.
Hence, $e(L,\cup_{i=1}^{s}D_i)\ge \sum_{i=1}^{s}\frac{3|D_i|}{2}+3$, which implies that there exists a cycle, say $D_1$, such that $e(L,D_1)\ge \frac{3|D_1|}{2}+1$.

First suppose that $|D_1|\ge 6$. Since $\{D_1, \ldots, D_s\}$ is a minimal system of $G$, we know $D_1$ is a minimal system of $G-\cup_{i=2}^s D_i$. Thus by Lemma \ref{lemma 1} we have $e(L, D_1)\le 2+2(\frac{|D_1|}{2})+(\frac{|D_1|}{2}-1)=\frac{3|D_1|}{2}+1$. Recall that $e(L,D_1)\ge \frac{3|D_1|}{2}+1$. It follows that $d_{D_1}(y_1)=d_{D_1}(y_3)=\frac{|D_1|}{2}-1$ and $d_{D_1}(y_2)=1$. Hence, we can easily get a feasible quadrilateral in $G[V(Q\cup D_1)]$ by Lemma \ref{lemma 1} (2), a contradiction.

Thus $|D_1|=4$ and $e(L, D_1)\ge 7$. Let $D_1=u_1u_2u_3u_4u_1$ with $u_1\in X$. Clearly, $e(Y_L, D_1)\ge 1$. Since $y_1$ and $y_3$ are symmetric, we may assume that $x_{2p}u_1\in E$ or $y_1u_1\in E$.

First claim that $e(X_L, D_1)\le 4$. For otherwise, $e(X_L, D_1)\ge 5$, then there exist vertex sets $\{v_2, v_4\}$ and $\{v_2^\prime, v_4^\prime\}$ such that $\{v_2, v_4\}=\{v_2^\prime, v_4^\prime\}=\{u_2, u_4\}$, $x_1v_2, x_{2p-1}v_2, y_2v_4\in E$ and $x_1v_2^\prime$, $y_2v_2^\prime$, $x_{2p-1}v_4^\prime \in E$.
If $y_1u_1\in E$, then $G[V(D_1\cup P\cup Q)]$ contains two disjoint feasible cycles $v_2P[x_1, x_{2p-1}]v_2$ and $y_1y_2v_4u_1y_1$, contradicting (\ref{eqpath3}). If $x_{2p}u_1\in E$, then $G[V(D_1\cup P)]$ contains disjoint subgraphs $D_1^\prime=x_{2p-1}x_{2p}u_1v_4^\prime x_{2p-1}$ and $P^\prime=u_3v_2^\prime x_1x_2$ such that $y_2v_2^\prime\in E$. By the choice of $P$ and $Q$, we have $y_2x_2\in E$. It follows that $G[V(D_1\cup P\cup Q)]$ contains two disjoint feasible cycles $x_{2p-1}x_{2p}u_1v_4^\prime x_{2p-1}$ and $x_1x_2y_2v_2^\prime x_1$, a contradiction again.

Then claim that $d_{D_1}(y)\le 1$ for each $y\in Y_L$.
For otherwise, $d_{D_1}(y)=2$ for some $y\in Y_L$. By (\ref{eqpath3}) ,
%
%
%
we have $N_{D_1}(u)\cap N_{D_1}(v)=\emptyset$ for each subset $\{u,v\}\subseteq X_L$, which implies that $e(X_L, D_1)\le 2$ and so $e(Y_L, D_1)\ge 5$. Thus we have $x_{2p}v_1, y_1v_3, y_3v_3\in E$, where $\{v_1, v_3\}=\{u_1, u_3\}$. It follow that $G[V(D_1\cup P\cup Q)]$ contains a feasible quadrilateral $v_3y_1y_2y_3v_3$ and a path $P^\prime=Pv_1u_4$, a contradiction.

So we have $7\le e(L, D_1)\le 4+3\times 1=7$. It follows that $e(X_L, D_1)=4$ and $d_{D_1}(y)=1$ for each $y\in Y_L$.
If $N_{D_1}(y_1)\cap N_{D_1}(y_3)\ne \emptyset$, then $G[V(D_1\cup P\cup Q)]$ contains a feasible quadrilateral $D_1^\prime$ and a path $P^\prime$ disjoint from $D_1^\prime$ such that $|P^\prime|$ is even and $|S_{P^\prime}|=3$ as $e(\{x_1, x_{2p-1}\}, D_1)\ne 0$,  a contradiction. 
Then $N_{D_1}(y_1)\cap N_{D_1}(y_3)= \emptyset$ and assume that $x_{2p}u_1, y_1u_1, y_3u_3\in E$. Recall that $e(X_L, D_1)=4$, we have $d_{D_1}(x_1)=2$ or $d_{D_1}(y_2)\ge 1$ and assume $y_2u_2\in E$. In both cases, $G[V(D_1\cup P\cup Q)]$ contains a feasible quadrilateral $D_1^\prime$ and a path $P^\prime$ of even order with $|S_{P^\prime}|=3$, where $D_1^\prime=x_1u_2u_3u_4x_1$ and $P^\prime=x_{2p-1}x_{2p}u_1y_1y_2y_3$, or $D_1^\prime=y_2y_3u_3u_2y_2$ and $P^\prime=Pu_1y_1$, a contradiction.
\end{proof}

Recall that $G$ is $S$-$(k-1)$-feasible, let $\{C_1, \ldots, C_{k-1}\}$ be a minimal $S$ system of $G$. Then by Lemma \ref{lemma 3.1} we have $|S_{C_i}|=2$ for each $i\in \{1,\ldots, k-1\}$. Let $H=G-\cup_{i=1}^{k-1}C_i$. Then $|S_H|\ge 3$ as $|S|\ge 2k+1$. Now we divide the proof of Theorem \ref{theorem 2} into the following two cases.

\noindent\textbf{Case 1.} $|S_H|\ge 4$ or $|X_H|>|S_H|=3$.

Choose $\{C_1, \ldots, C_{k-1}, H\}$ such that $\{C_1, \ldots, C_{k-1}\}$ is a minimal $S$ system of $G$ and $\max\{d_H(u): u\in S_H\}$ is maximum.
First suppose that $\max\{d_H(u): u\in S_H\}\ge 2$. Let $y_2$ be the vertex in $S_H$ such that $d_H(y_2)\ge 2$ and $y_1, y_3\in N_H(y_2)$. Let $Q=y_1y_2y_3$. Since $G$ is not $S$-$k$-feasible, by Lemma \ref{lemma 3.2} (2), $G-Q$ contains a minimal $S$ system $\{C_1^\prime, \ldots , C_{k-1}^\prime\}$ and $H^{\prime}$ has a $l$-$S$-matching $M$, where $H^{\prime}=G-Q-\cup _{i=1}^{k-1} C_i^\prime$ and $l\ge \min\{|Y_{H^{\prime}}|, |S_{H^{\prime}}|\}=\min\{(4-2, 3-1\}=2$.
 Let $\{x_1x_2, x_3x_4\}\subseteq M$ with $x_1, x_3\in S$. By Claim \ref{claim xin2}, $H^\prime$ does not contain a path $P$ of even order such that $|S_P|=2$. 
Thus $E(G[V(M)])=M$, $e(\{x_1, x_3\}, H^\prime-M)\le |Y_{H^\prime-M}|$ and if $e(\{x_1, x_3\}, H^\prime-M)\ne 0$ then $e(\{x_2, x_4\}, H^\prime-M)\le |X_{H^\prime-M}|$. And we also have $e(\{x_1x_2, x_3x_4\}, Q)\le 2$ for otherwise $G-\cup _{i=1}^{k-1} C_i^\prime$ is $S$-$1$-feasible and so $G$ is $S$-$k$-feasible. It follows that 
\begin{align*}
  e(\{x_1x_2, x_3x_4\}, \cup_{i=1}^{k-1} C_i^\prime)  & =  e(\{x_1x_2, x_3x_4\}, G-Q-M-(H^\prime-M))  \\
    &  \ge  2(n+1)-2-4-\max\{|Y_{H^\prime-M}|+|X_{H^\prime-M}|, 2|X_{H^\prime-M}|\} \\
   & \mbox{\ $\ge\sum_{i=1}^{k-1}|C_i^\prime|+2$. \ }
\end{align*}
Thus by Lemmas \ref{lemma 1} (4) and \ref{lemma 2} (2), there exists a feasible quadrilateral, say $C_1$, such that $G[V(C_1\cup \{x_1x_2,x_3x_4\})]$ contains a feasible quadrilateral and a path $P^\prime$ of order 4 with $|S_{P^\prime}|=2$. So by Claim \ref{claim xin2}, $G$ is $S$-$k$-feasible, a contradiction.

Thus $\max\{d_H(u): u\in S_H\}\le 1$. By Lemma \ref{lemma 3.2} (1), $G$ contains a minimal $S$ system $\{C_1^\prime, \ldots , C_{k-1}^\prime\}$ and $H^{\prime}$ has a $|S_{H^\prime}|$-$S$-matching $M$, where $H^{\prime}=G-\cup _{i=1}^{k-1} C_i^\prime$.
Clearly, $E(G[V(M)])=M$ and $e(X_M, H^\prime-M)=0$. It follows that 
\begin{equation*}
\mbox{\ $e(M, \cup _{i=1}^{k-1} C_i^\prime)\ge |X_M|(n+1)-|Y_M|(\frac{|H^\prime-M|}{2})-2|X_M|=|X_M|(\sum_{i=1}^{k-1}\frac{|C_i^\prime|}{2}+(\frac{|V(M)|-4}{2})+1)$. \ } 
\end{equation*}
Thus by Lemmas \ref{lemma 1} (4) and \ref{lemma 2} (2), $G$ contains a minimal $S$ system $\{C_1^0, \ldots , C_{k-1}^0\}$ and $\max\{d_{H^0}(u): u\in S_{H^0}\}\ge 2$, where $H^0=G-\cup_{i=1}^{k-1}C_i^0$, a contradiction.

\noindent\textbf{Case 2.} $|X_H|=|S_H|=3$.

Choose $\{C_1, \ldots, C_{k-1}, H\}$ such that $\{C_1, \ldots, C_{k-1}\}$ is a minimal $S$ system of $G$ and $H$ contains a $3$-$S$-matching $M$, and subject to this, $e(H)$ is maximum. By Lemma \ref{lemma 3.2} (1), such $H$ exists. Let $M=\{x_1y_1, x_2y_2, x_3y_3\}$ such that $x_1, x_2, x_3$ in the same partition class. By Claim \ref{claim xin1} and the fact that $H$ is not $S$-1-feasible, $H$ is isomorphic to a graph of $\{H_1, H_2, H_3\}$, where $V(H_1)=V(H_2)=V(H_3)=V(M)$, $E(H_1)=M$, $E(H_2)=M\cup \{y_1x_2\}$ and $E(H_3)=M\cup \{y_1x_2, x_2y_3\}$.

First suppose that $H\cong H_1$ or $H\cong H_2$. Clearly, $x_1y_2, x_2y_3, x_3y_1\notin E$. Thus $e(H, \cup_{i=1}^{k-1}C_i)\ge 3(n+1)-8=\sum_{i=1}^{k-1}\frac{3|C_i|}{2}+4$ as $e(H)\le 4$. This implies that there exists a cycle, say $C_1$, such that $e(H,C_1)\ge \frac{3|C_1|}{2}+1$.
By Lemma \ref{lemma 1} (4), Lemma \ref{lemma 2} (2) and the choice of $H$, we have $H\cong H_2$ and $|C_1|=4$.
Let $C_1=u_1v_1u_2v_2u_1$ with $u_1$ and $x_1$ in the same partition class.

If $e(x_1y_1x_2y_2, C_1)\ge 6$, then by Lemma \ref{lemma 2} (3) and the assumption that $G$ is not $S$-$k$-feasible, $e(x_1y_1x_2y_2, C_1)=6$, $d_{C_1}(x_1)=0$ or $d_{C_1}(y_2)=0$, and $e(\{x_3,y_3\}, C_1)\ge 1$ as $e(H,C_1)\ge 7$, which implies that $e(\{x_i, y_i\}, C_1)=4$ and $e(\{x_3,y_3\}, \{u_j, v_j\})\ge 1$ for some $i, j\in \{1,2\}$. It follows that $G[V(C_1\cup H)]$ contains a quadrilateral $C^\prime=x_iv_{3-j}u_{3-j}y_ix_i$ and a subgraph $H^\prime$ such that $H^\prime$ contains a 3-$S$-matching, $e(H^\prime)\ge 5$ and $H^\prime$ is disjoint from $C^\prime$, a contradiction.

 Thus $e(x_1y_1x_2y_2, C_1)\le 5$ and $e(\{x_3,y_3\}, C_1)\ge 2$. 
Note that $e(\{x_1,x_2,x_3\}, C_1)\ge 4$ or $e(\{y_1,y_2,y_3\}, C_1)\ge 4$. If $d_{C_1}(x_3)\ge 1$ and $d_{C_1}(y_3)\ge 1$, then $e(\{x_1,x_2\}, \{v_i\})\ge 1$ and $x_3v_j\in E$ for $\{i,j\}=\{1,2\}$; If $d_{C_1}(x_3)=0$ or $d_{C_1}(y_3)=0$, then $e(x_1y_1x_2y_2, C_1)=5$, which implies that $e(\{x_i, y_i\}, \{u_j, v_j\})=2$ and $e(\{x_{3-i}, y_{3-i}, x_3, y_3\}, \{u_{3-j}, v_{3-j}\})\ge 2$ for some $i, j\in \{1,2\}$. In both cases, $G[V(C_1\cup H)]$ contains two disjoint subgraphs $C^\prime$ and $H^\prime$ such that $C^\prime$ is a quadrilateral, $H^\prime$ contains a 3-$S$-matching and $e(H^\prime)\ge 5$, a contradiction.

Then suppose that $H\cong H_3$. Let $L_1=\{x_1,y_1,x_3,y_3\}$ and $L_2=\{x_1,y_1,y_2,x_3\}$. Note that $x_1y_3, y_1x_3, x_1y_2\notin E$, $e(L_1,H)=6$ and $e(L_2,H)=5$, we have $e(L_l,\cup_{i=1}^{k-1}C_i)\ge 2(n+1)-6= \sum_{i=1}^{k-1}|C_i|+2$ for each $l\in \{1,2\}$, which implies that there exist two cycles, say $C_i$ and $C_j$ (maybe $i=j$), such that $e(L_1, C_i)\ge |C_i|+1$ and $e(L_2, C_j)\ge |C_j|+1$. By Lemma \ref{lemma 1} (3) and (4) and Claim \ref{claim xin1}, we have $|C_i|=|C_j|=4$, and for each $l\in \{1,2\}$,
\begin{equation}\label{eqpath6}
\mbox{\ $G[V(C_l\cup H)]$ does not contain disjoint cycle $C$ and path $P$, where $|C|=4$ and $|P|=6$.\ }
\end{equation} 

First consider the subgraph $G[V(C_i)\cup L_1]$. Since $e(L_1, C_i)\ge 5$, we have $e(\{x_p,y_p\},C_i)\ge 3$ and $e(\{x_q,y_q\},C_i)\ge 1$, where $\{p, q\}=\{1, 3\}$.
If $d_{C_i}(x_p)=2$, then $e(\{y_p,y_q\},C_i)\ge 3$ or $d_{C_i}(x_q)\ge 1$, which implies that $G[V(C_i\cup H)]$ contains two disjoint quadrilaterals or a quadrilateral and a path of order 6, so $G$ is $S$-$k$-feasible by Claim \ref{claim xin1}, a contradiction. Thus $d_{C_i}(x_p)=1$ and $d_{C_i}(y_p)=2$. If $d_{C_i}(x_q)\ne 0$, then $G[V(C_i\cup H)]$ contains a quadrilateral and a path of order 6, contradicting (\ref{eqpath6}). Thus $d_{C_i}(x_q)=0$ and $d_{C_i}(y_q)=2$ as $e(L_1, C_i)\ge 5$.

Then consider the subgraph $G[V(C_j)\cup L_2]$. Since $e(L_2,C_j)\ge 5$ and $e(\{x_1,x_3\},C_j)\ge 1$, we have $d_{C_j}(y_2)\le 1$ by (\ref{eqpath6}). Thus $e(\{x_1,y_1\},C_j)\ge 2$. 
If $d_{C_j}(x_1)\ge 1$ and $d_{C_j}(y_1)\ge 1$, then $G[V(C_j\cup H)]$ contains a quadrilateral $C$ and a path of order 6 which disjoint from $C$ as $e(\{y_1,y_2\},C_j)\ge 3$ or $e(\{x_1,x_3\},C_j)\ge 3$, contradicting (\ref{eqpath6}). If $d_{C_j}(x_1)=0$, then $d_{C_j}(y_2)=1$ and $d_{C_j}(y_1)=d_{C_j}(x_3)=2$, thus $G[V(C_j\cup H)]$ contains two disjoint quadrilaterals, a contradiction. Thus $d_{C_j}(y_1)=0$, $d_{C_j}(y_2)=1$ and $d_{C_j}(x_1)=d_{C_j}(x_3)=2$.

By the argument above, we have $i\ne j$. Let $C_i=u_1u_2u_3u_4u_1$ and $C_j=v_1v_2v_3v_4v_1$ with $u_1, v_1, x_1$ in the same partition class. We may assume that $x_pu_2, y_2v_1\in E$. Then $G[V(H\cup C_i\cup C_j)]$ contains two disjoint quadrilaterals $y_qu_1u_4u_3y_q$, $x_qv_2v_3v_4x_q$ and a path $v_1y_2x_2y_px_pu_2$, thus $G$ is $S$-$k$-feasible by Claim \ref{claim xin1}, a contradiction.

The proof of Theorem \ref{theorem 2} is now complete.

\section{Proof of Theorem \ref{theorem}}
\label{pftheorem2}

In order to prove Theorem \ref{theorem}, first we show the following claim. Here, a path $P$ of $G$ is \textit{good} if $|S_P|\ge 2$, $|P|$ is even and $G-P$ has a $|S_{G-P}|$-$S$-matching. A good path $P$ is maximal if there does not exist a good path $Q$ with $S_P\subset S_Q$. 

\begin{clm}\label{keyclaim}
Let $D_1, D_2, \ldots, D_s$ be $s$ disjoint feasible cycles in $G$. If $G-\cup_{i=1}^sD_i$ has a maximal good path $P=x_1y_1\cdots x_py_p$ with $x_1, x_p\in S$, then $G[V(\cup_{i=1}^sD_i\cup P)]$ is $S_{\cup_{i=1}^sD_i\cup P}$-$s$-cyclable or $G[V(P)]$ is $S_P$-1-cyclable. 
\end{clm}
\begin{proof}
Let $H=G-\cup_{i=1}^sD_i$, $M$ a $|S_{H-P}|$-$S$-matching of $H-P$, and let $H^\prime=H-P-M$. If $d_{H^\prime}(x_1)\ne 0$, then assume that $v\in V(H^\prime)$ is a vertex such that $x_1v\in E$. Since $P$ is maximal, we have $e(\{x_1,y_p\}, M)=0$ and $e(\{x_p, v\}, M)=0$ if $v$ exists. If $d_{H^\prime}(x_1)=0$, let $L=\{x_1, y_p\}$, otherwise, let $L=\{x_1, v, x_p, y_p\}$. 
If $x_1y_p\in E$ or $x_pv\in E$, then $G[V(P)]$ is $S_P$-1-cyclable. Now suppose $x_1y_p, x_pv\notin E$, which implies that $e(L, G-M)\ge \frac{|L|}{2}(n+1)$ as $e(L,M)=0$. If $d_{H^\prime}(x_1)=0$, then $e(L,H^\prime)\le \frac{|H^\prime|}{2}$ and thus $e(L, \cup_{i=1}^{s}D_i\cup P)\ge \frac{1}{2}\sum_{i=1}^{s}|D_i|+\frac{|P|}{2}+1$; if $d_{H^\prime}(x_1)\ne 0$, then $e(L,H^\prime-\{v\})\ge |H^\prime|$ or $e(L, \cup_{i=1}^{s}D_i\cup (P\cup \{v\}))\ge \sum_{i=1}^{s}|D_i|+(|P|+2)+1$. In the above two cases, by Lemma \ref{lemma 4}, we have $G[V(\cup_{i=1}^sD_i\cup P)]$ is $S_{\cup_{i=1}^sD_i\cup P}$-$s$-cyclable or $G[V(P)]$ is $S_P$-1-cyclable .
\end{proof}
%

Now we divide the proof into two steps. The first step is to show that $G$ contains $s$ disjoint feasible cycles covering $S$ for some $s\ge k$. The second step is to transform these $s$ disjoint feasible cycles into exactly $k$ disjoint feasible cycles, that is, $G$ is $S$-$k$-cyclable. 

\subsection{Showing that $G$ is $S$-$s$-cyclable for some $s\ge k$}

Suppose that $G$ is not $S$-$s$-cyclable for all $s\ge k$.
Let $r$ be the largest integer such that $G$ is $S$-$r$-feasible. Let $\{C_1,\ldots, C_r\}$ be a minimal system of $G$.
By Theorem \ref{theorem 2} we have $r\ge k+1$ if $|S|\ge 2k+3$ and $r=k$ if $|S|=2k+2$. Let $H=G-\cup_{i=1}^rC_i$. Clearly, by our assumption $|S_H|\ge 1$.


%
%
%
%



\noindent\textbf{Case 1.} $r=k$ and $|S|=2k+2$.

Choose $C_1, \ldots, C_k$ such that $\{C_1, \ldots, C_k\}$ is a minimal $S$ system of $G$. By Lemmas \ref{lemma 3.1}-\ref{lemma 3.2} and $|H|\ne 0$, we can choose $C_1, \ldots, C_k$ such that $|S_{C_i}|=2$ for each $i\in \{1, \ldots, k\}$ and $H$ has a $|S_H|$-$S$-matching $M$. Clearly, $|S_H|=2$. Let $S_H=\{x, x^\prime\}$ and $M=\{xy, x^\prime y^\prime\}$. 

Applying Claim \ref{keyclaim} with $s=k$ and $(D_1,\ldots, D_s)=(C_1,\ldots, C_k)$, we know that $H$ does not contain a path $P$ of even order with $|S_P|=2$, for otherwise, $G$ is $S$-$k$-cyclable or $S$-$(k+1)$-cyclable, a contradiction.
Thus $xy^\prime, x^\prime y\notin E$ and $N_{H-M}(x)\cap N_{H-M}(x^\prime)=\emptyset$, and if $N_{H-M}(y)\cap N_{H-M}(y^\prime)\ne \emptyset$ then $N_{H-M}(x)\cup N_{H-M}(x^\prime)=\emptyset$. It follows that $e(M, H)\le 4+|H-M|= |H|$ and so $e(M, \cup_{i=1}^k C_i)\ge 2(n+1)-|H|= \sum_{i=1}^k |C_i|+2$. Hence, there exists a cycle, say $C_1$, such that $e(M, C_1)\ge |C_1|+1$. 

If $|C_1|\ge 6$, then by Lemma \ref{lemma 1} (4) and the choice of $C_1, \ldots, C_k$ we have $e(M, C_1)\le |C_1|$, a contradiction. If $|C_1|=4$, then by Lemma \ref{lemma 2} (2), $G[V(C_1\cup M)]$ contains a quadrilateral $C^\prime$ and a path $P^\prime$ of order 4 such that $P^\prime$ is disjoint from $C^\prime$. Thus by Claim \ref{keyclaim}, $G$ is $S$-$k$-cyclable or $S$-$(k+1)$-cyclable, a contradiction.



\noindent\textbf{Case 2.} $r\ge k+1$ and $|S|\ge 2k+3$.

By Lemma \ref{lemma 3.2} (1), we can choose $C_1, \ldots, C_{r}$ in $G$ such that $H$ has a $|S_H|$-$S$-matching $M$ and subject to this, $\sum_{i=1}^{r}|S_{C_i}|$ is maximum.  
Applying Claim \ref{keyclaim} with $s=r$ and $(D_1,\ldots, D_s)=(C_1,\ldots, C_r)$, we know that $H$ does not contain a good path $P$, for otherwise, we can extend $P$ to a maximal one, and then $G$ is $S$-$r$-cyclable or $S$-$(r+1)$-cyclable, a contradiction. 

\noindent\textbf{Case 2.1.} $|S_H|\ge 2$.

Recall that $H$ does not contain a good path. We have $E(G[V(M)])=M$, and for each pair $xy, x^\prime y^\prime\in M$ with $x, x^\prime\in X$ we have $N_{H-M}(x)\cap N_{H-M}(x^\prime)=\emptyset$, and $N_{H-M}(x)\cup N_{H-M}(x^\prime)=\emptyset$ if $N_{H-M}(y)\cap N_{H-M}(y^\prime)\ne \emptyset$. It follows that $xy^\prime, x^\prime y\notin E$ and $e(\{x,y,x^\prime, y^\prime\}, H)\le 4+|H-M|\le |H|$. Thus $e(\{x,y,x^\prime, y^\prime\}, \cup_{i=1}^r C_i)\ge 2(n+1)-|H|= \sum_{i=1}^r |C_i|+2$. It follows that there exists a cycle, say $C_1$, such that $e(\{x,y,x^\prime, y^\prime\}, C_1)\ge |C_1|+1$. Hence, by Lemma \ref{lemma 4}, $G$ contains $r$ disjoint feasible cycles $C_1^\prime, \ldots, C_r^\prime$ such that $H^\prime$ has a $|S_{H^\prime}|$-$S$-matching and $\sum_{i=1}^r|S_{C_i^\prime}|>\sum_{i=1}^r|S_{C_i}|$, where $H^\prime=G-\cup_{i=1}^rC_i^\prime$, a contradiction.

\noindent\textbf{Case 2.2.} $|S_H|=1$.

Let $M=xy$ with $x\in S$. If there exists a path $P^0$ of even order covering $S_{H\cup C_i}$ for some $i\in \{1,\ldots,r\}$, applying Claim \ref{keyclaim} with $s=r-1$ and $(D_1,\ldots, D_s, P)=(C_1,\ldots, C_{i-1},C_{i+1}, \ldots,   C_r, P^0)$, then $G$ is $S$-$(r-1)$-cyclable or $S$-$r$-cyclable, a contradiction.


Then $e(\{x,y\},\cup_{i=1}^r C_i)=0$ and for each $i\in \{1,\ldots,r\}$ and $uv\in E(C_i)$ with $u\in S$ we have $N_{H-M}(x)\cap N_{H-M}(u)=\emptyset$ and if $e(\{x, u\}, H-M)\ne 0$ then $N_{H-M}(y)\cap N_{H-M}(v)=\emptyset$. Note that $xv, yu\notin E$, we have 
\begin{equation*}
\mbox{\ $2(n+1)\le e(\{x, y, u, v\}, \cup_{i=1}^r C_i\cup M\cup (H-M))\le (0+\sum_{i=1}^r|C_i|)+2+|H-M|=2n$, \ }
\end{equation*}
 a contradiction.

\subsection{Showing that $G$ is $S$-$k$-cyclable}

Note that $G$ contains $s$ disjoint feasible cycles $C_1,\ldots, C_s$ covering $S$ for some $s\ge k$. Choose $C_1,\ldots, C_s$ such that $s\ge k$ and $s$ is minimum, and subject to this, $\sum_{i=1}^s|C_i|$ is minimum. Let $H=G-\cup_{i=1}^sC_i$. If $s=k$, then we are done. So we may assume that $s\ge k+1$. Next we get a contradiction by showing that $G$ is $S$-$(s-1)$-cyclable. To do this, we first study the structure between $C_i$ and $\cup_{j=1}^sC_j-C_i$ and give the following four claims, where $1\le i\le s$.

%
%
%

\begin{clm}\label{claim 4.1}
$e(C_i, \cup_{j\in \{1, \ldots, s\}\setminus \{i\}}C_j)\ne 0$ for each $i\in \{1, \ldots, s\}$.
\end{clm}
\begin{proof}
On the contrary, suppose that $e(C_1, \cup_{j=2}^sC_j)=0$. Then 
\begin{equation}\label{eqclaim4.1}
 \begin{split}
& \mbox{\ for any $a\in V(C_1)$ and $b\in V(\cup_{j=2}^sC_j)$,\ }\\
& \mbox{\ $e(\{a, b\}, \cup_{j=1}^sC_j)\le \frac{1}{2}\sum\nolimits_{j=1}^s|C_j| $ and $ab\notin E$.\ } 
  \end{split}
\end{equation}
In particular, by the condition $\sigma_{1,1}(S)\ge n+1$,
\begin{equation}\label{eq2claim4.1}
 \begin{split}
& \mbox{\ if $a\in S_{C_1}, b\in Y_{\cup_{j=2}^sC_j}$ or $a\in Y_{C_1}, b\in S_{\cup_{j=2}^sC_j}$,\ }\\
& \mbox{\ then $e(\{a, b\}, H)\ge \frac{|H|}{2}+1$ and $d_H(a)\ge 1, d_H(b)\ge 1$.\ }
  \end{split}
\end{equation}
\begin{center}
\begin{picture}(382,110)\linethickness{0.8pt}
\put(31.3,75.3){\circle{61.6}}
\Line(56,93.6)(78.3,105)
\Line(78.3,105)(100.7,93.6)

\put(125.3,75.3){\circle{61.6}}
\put(221.3,75.3){\circle{61.6}}

\put(315.3,75.3){\circle{61.6}}

\put(8.2,54.9){\vector(-68,73){1}}
\put(148.4,95.6){\vector(68,-73){1}}
\put(198.5,54.7){\vector(-68,73){1}}
\put(337.7,96.3){\vector(68,-73){1}}
\put(10,60){$C_1$}
\put(135,60){$C_2$}
\put(200,60){$C_1$}
\put(325,60){$C_2$}
%
\put(42,89){$x$}
\put(46.3,76.5){$x^+$}
\put(43,58){$x_1$}
\put(101.8,76.5){$z^-$}
\put(105.3,89){$z$}
\put(103,58){$z_1$}
\put(76,108.5){$v$}
%
\put(380,59){$\in X$}
\put(380,31){$\in S$}
\put(380,45){$\in Y$}
\put(370,62){\circle*{3.8}}
\put(370,48){\circle*{3.8}\color{white}\circle*{2.6}}
\color{black}\polygon*(369,32.5)(371.5,32.5)(371.5,35)(369,35)\polygon(369,32.5)(371.5,32.5)(371.5,35)(369,35)
\Line(246,93.6)(261,105)
\Line(275.7,105)(290.7,93.6)
\put(261,105){\circle*{3.8}\color{white}\circle*{2.6}}
\put(275.7,105){\circle*{3.8}\color{white}\circle*{2.6}}
\put(258,108.5){$v_1$}
\put(273,108.5){$v_2$}
\put(233,89){$x_1$}
\color{black}\polygon*(245,92)(247.5,92)(247.5,94.5)(245,94.5)\polygon(245,92)(247.5,92)(247.5,94.5)(245,94.5)
\put(237.5,76.5){$x_1^+$}
\put(251.6,80.9){\circle*{3.8}\color{white}\circle*{2.6}}
\put(295.3,90){$z_1$}
\color{black}\polygon*(289,92)(291.5,92)(291.5,94.5)(289,94.5)\polygon(289,92)(291.5,92)(291.5,94.5)(289,94.5)
\put(290,76){$z_1^-$}
\put(285,80.9){\circle*{3.8}\color{white}\circle*{2.6}}

%
\put(56,93.6){\circle*{3.8}}
\put(61.6,80.9){\circle*{3.8}\color{white}\circle*{2.6}}
\put(100.7,93.6){\circle*{3.8}}
\put(78.3,105){\circle*{3.8}\color{white}\circle*{2.6}}
\put(94.9,80.3){\circle*{3.8}\color{white}\circle*{2.6}}

\color{black}\polygon*(56,58)(58.5,58)(58.5,60.5)(56,60.5)\polygon(56,58)(58.5,58)(58.5,60.5)(56,60.5)
\color{black}\polygon*(98,58)(100.5,58)(100.5,60.5)(98,60.5)\polygon(98,58)(100.5,58)(100.5,60.5)(98,60.5)

\put(23,24){\footnotesize{The case $N_H(x)\cap N_H(z)\ne \emptyset$}}
\put(210,24){\footnotesize{The case $N_H(x_1)\cap N_H(z_1)=\emptyset$}}
\end{picture}

\vspace{-15pt}
\footnotesize{Figure 1. The graph $G[V(C_1\cup C_2\cup H)]$}
\end{center}

First suppose that there exist two vertices $x\in X_{C_1}$ and $z\in X_{\cup_{j=2}^sC_j}$ such that $N_H(x)\cap N_H(z)\ne\emptyset$ (see Figure 1). Assume that $z\in X_{C_2}$ and $v\in N_H(x)\cap N_H(z)$.
Let $x_1$ and $z_1$ be the vertices such that $S_{C_1[x^+, x_1]}=\{x_1\}$ and $S_{C_2[z_1, z^-]}=\{z_1\}$.
By (\ref{eq2claim4.1}), we have $e(\{x_1, x^+, z_1, z^-\},H)\ge |H|+2$, which implies that there exists a vertex $u\in V(H)\setminus \{v\}$ such that $u\in N_H(x_1)\cap N_H(z_1)$ or $u\in N_H(x^+)\cap N_H(z^-)$. It follows that $G[V(C_1\cup C_2\cup \{v, u\})]$ is $S_{C_1\cup C_2}$-1-cyclable, this is contrary to the choice of $s$.

Then for any pair of vertices $x\in X_{C_1}$ and $z\in X_{\cup_{j=2}^sC_j}$ we have $N_H(x)\cap N_H(z)=\emptyset$.
Let $x_1\in S_{C_1}$ and $z_1\in S_{C_2}$ (see Figure 1). Then $N_H(x_1)\cap N_H(z_1)=\emptyset$. Let $v_1$ and $v_2$ be the vertices such that $v_1\in N_H(x_1)$ and $v_2\in N_H(z_1)$, and $d_{\cup_{j=2}^sC_j}(v_1)=0$ and $d_{C_1}(v_2)=0$. By (\ref{eq2claim4.1}), $v_1$ and $v_2$ exist.
Let $L_1=\{x_1, z_1, v_1, v_2\}$ and $L_2=\{x_1, z_1, x_1^+, z_1^-\}$. Then $e(\{x_1,z_1\}, H)\le \frac{|H|}{2}$, $e(\{v_1, v_2\}, \cup_{j=1}^sC_j)\le \frac{1}{2}\sum_{j=1}^s|C_j|$ and $x_1v_2, z_1v_1\notin E$. Hence, by (\ref{eqclaim4.1}) and (\ref{eq2claim4.1}), we have \begin{align*}
   e(\{v_1, v_2\}, H) &  =e(L_1,G)-e(\{x_1,z_1\},G)-e(\{v_1,v_2\}, \cup_{j=1}^sC_j) \\
    & \mbox{\  $\ge 2(n+1)-n-\frac{1}{2}\sum_{j=1}^s|C_j|=\frac{|H|}{2}+2$,\ }
\end{align*}
\begin{equation*}
\mbox{\ and $e(\{x_1^+, z_1^-\}, H)=e(L_2,H)-e(\{x_1,z_1\},H)\ge 2(\frac{|H|}{2}+1)-\frac{|H|}{2}=\frac{|H|}{2}+2$. \ }
\end{equation*}
This implies that there exist two distinct vertices $u_1, u_2$ such that $u_1\in N_H(v_1)\cap N_H(v_2)$ and $u_2\in N_H(x_1^+)\cap N_H(z_1^-)$. It follows that $G[V(C_1\cup C_2\cup \{v_1,v_2,u_1, u_2\})]$ is $S_{C_1\cup C_2}$-1-cyclable, a contradiction.
\end{proof}

\begin{clm}\label{claim 4.2}
Let $C_i, C_j$ with $i,j\in \{1,\ldots, s\}$ and $i\ne j$ and $xy\in E(X_{C_i}, Y_{C_j})$. Then $y^-, y^+\in S$ and $d_H(y^-)=d_H(y^+)=0$, where $y^-, y^+$ are the predecessor and successor of $y$ in $C_j$.
\end{clm}
\begin{proof}
Suppose that there exists an edge $x_1y_2\in E(X_{C_1}, Y_{C_2})$ such that $y_2^-\notin S$ or $d_H(y_2^-)\ne0$ (See Figure 2). Let $x_1^\prime$ and $x_2$ be the vertices such that $S_{C_1[x_1^+, x_1^\prime]}=\{x_1^\prime\}$ and $S_{C_2[x_2, y_2]}=\{x_2\}$. Let $v$ be the vertex such that $v=x_2^+$ if $y_2^-\notin S$, and $v\in N_H(y_2^-)$ if $y_2^-\in S$ ($y_2^-=x_2$). 
\begin{center}
\begin{picture}(382,110)\linethickness{0.8pt}
\put(31.3,75.3){\circle{61.6}}
\Line(56,93.6)(100.7,93.6)
\put(125.3,75.3){\circle{61.6}}
\put(221.3,75.3){\circle{61.6}}
\Line(246,93.6)(290.7,93.6)
\put(315.3,75.3){\circle{61.6}}
\Line(284.5,79.5)(275.3,61.8)
\put(8.2,54.9){\vector(-68,73){1}}
\put(148.4,95.6){\vector(68,-73){1}}
\put(198.5,54.7){\vector(-68,73){1}}
\put(337.7,96.3){\vector(68,-73){1}}
\put(10,60){$C_1$}
\put(135,60){$C_2$}
\put(200,60){$C_1$}
\put(325,60){$C_2$}
%
\put(42,89){$x_1$}
\put(46.3,76.5){$x_1^+$}
\put(43,58){$x_1^{\prime}$}
\put(101.8,76){$y_2^-$}
\put(105.3,90){$y_2$}
\put(109.5,49.5){$x_2$}
\put(102,60){$v$}
\put(110,60){$(x_2^+)$}
\put(380,59){$\in X$}
\put(380,31){$\in S$}
\put(380,45){$\in Y$}
\put(370,62){\circle*{3.8}}
\put(370,48){\circle*{3.8}\color{white}\circle*{2.6}}
\color{black}\polygon*(369,32.5)(371.5,32.5)(371.5,35)(369,35)\polygon(369,32.5)(371.5,32.5)(371.5,35)(369,35)
\put(237.5,76.5){$x_1^+$}
\put(233,89){$x_1$}
\put(303,75){$(x_2)$}
\put(273.5,50.9){$v$}
\put(295.3,90){$y_2$}
\put(290,76){$y_2^-$}
\put(233,58){$x_1^{\prime}$}
\put(56,93.6){\circle*{3.8}}
\put(61.6,80.9){\circle*{3.8}\color{white}\circle*{2.6}}
\put(100.7,93.6){\circle*{3.8}\color{white}\circle*{2.6}}
\put(94.9,80.3){\circle*{3.8}}
\put(97,63){\circle*{3.8}\color{white}\circle*{2.6}}
\put(275.3,61.8){\circle*{3.8}\color{white}\circle*{2.6}}
\put(246,93.6){\circle*{3.8}}
\put(251.6,80.9){\circle*{3.8}\color{white}\circle*{2.6}}
\put(290.7,93.6){\circle*{3.8}\color{white}\circle*{2.6}}
\color{black}\polygon*(246,57.8)(248.5,57.8)(248.5,60.3)(246,60.3)\polygon(246,57.8)(248.5,57.8)(248.5,60.3)(246,60.3)
\color{black}\polygon*(56,58)(58.5,58)(58.5,60.5)(56,60.5)\polygon(56,58)(58.5,58)(58.5,60.5)(56,60.5)
\color{black}\polygon*(104.5,49.8)(107,49.8)(107,52.3)(104.5,52.3)\polygon(104.5,49.8)(107,49.8)(107,52.3)(104.5,52.3)
\color{black}\polygon*(283.5,78.5)(286,78.5)(286,81)(283.5,81)\polygon(283.5,78.5)(286,78.5)(286,81)(283.5,81)
\put(45,24){\footnotesize{The case $y_2^-\notin S$}}
\put(239,24){\footnotesize{The case $y_2^-\in S$}}
\end{picture}

\vspace{-15pt}
\footnotesize{Figure 2. The graph $G[V(C_1\cup C_2)\cup \{v\}]$}
\end{center}

Let $L=\{x_1^\prime, x_1^+, x_2, v\}$. 
First we calculate the upper bounds of $e(L, H)$ and $e(L, C_1\cup C_2)$. 
Consider paths $P_1=C_1[x_1^+, x_1]C_2[y_2, x_2]$, $Q_1=C_2[x_2^+, y_2^-]$, $P_2=C_1[x_1^\prime, x_1]C_2[y_2, x_2]v$ and $Q_2=C_1[x_1^+, {x_1^\prime}^-]\cup (C_2[x_2^+, y_2^-]-v)$. 
Clearly, if $y_2^-\notin S$, then $|C_1|+|C_2|=|P_1|+|Q_1|=|P_2|+|Q_2|$; otherwise, $|C_1|+|C_2|=|P_1|=|P_2|-1+|C_1[x_1^+, {x_1^\prime}^-]|$.
Thus by Lemma \ref{lemma 4} (2) and the choice of $s$, we have $x_1^\prime v, x_1^+x_2\notin E$ and $e(L, H)\le |H|$, and if $y_2^-\notin S$ then 
\begin{align*}
e(L, C_1\cup C_2) & =  e(\{x_1^+, x_2\}, P_1)+e(\{x_1^+, x_2\}, Q_1)+e(\{x_1^\prime, v\}, P_2)+e(\{x_1^\prime, v\}, Q_2) \\
 & \mbox{\ $\le  \frac{|P_1|}{2}+(0+\frac{|Q_1|}{2})+\frac{|P_2|}{2}+(\frac{|C_1[x_1^+, {x_1^\prime}^-]|+1}{2}+\frac{|C_2[x_2^+, y_2^-]-v|+1}{2})$ \ }\\
 & =   |C_1|+|C_2|+1,
\end{align*}
and if $y_2^-\in S$ then
\begin{align*}
e(L, C_1\cup C_2) & \le  e(\{x_1^+, x_2\}, P_1)+e(\{x_1^\prime, v\}, P_2)+e(\{x_1^\prime, v\}, C_1[x_1^+, {x_1^\prime}^-]) \\
 & \mbox{\ $\le  \frac{|P_1|}{2}+\frac{|P_2|}{2}+(\frac{|C_1[x_1^+, {x_1^\prime}^-]|+1}{2}+0)$\ } \\
 & =   |C_1|+|C_2|+1.
\end{align*}
Thus $e(L,\cup_{i=3}^sC_i)\ge 2(n+1)-(|C_1|+|C_2|+1)-|H|=\sum _{i=3}^s|C_i|+1$. Moreover, by Lemma \ref{lemma 4} (1) and the choice of $s$, we have $e(\{x_1^\prime, x_1^+\}, \cup_{i=3}^sC_i)\le \frac{1}{2}\sum _{i=3}^s|C_i|$, and $e(\{x_2, v\}, \cup_{i=3}^sC_i)\le \frac{1}{2}\sum _{i=3}^s|C_i|$ if $y_2^-\notin S$. Hence, $y_2^-\in S$, $x_2=y_2^-$ and $e(\{x_2, v\}, \cup_{i=3}^sC_i)\ge \frac{1}{2}\sum _{i=3}^s|C_i|+1$. This implies that there exists a cycle, say $C_3$, such that $e(\{x_2, v\},C_3)\ge \frac{1}{2}|C_3|+1$. 

Let $x_2^\prime$ be the vertex such that $S_{C_2[x_2^\prime, x_2^-]}=\{x_2^\prime\}$ and $L_2=\{x_1^\prime, x_1^+, x_2^\prime, x_2^-\}$.
Consider paths $P_1^\prime=C_1[x_1^+, x_1]C_2[y_2, x_2^\prime]$, $Q_1^\prime=C_2[{x_2^\prime}^+, y_2^-]$, $P_2^
\prime=C_1[x_1^\prime, x_1]C_2[y_2, x_2^-]$ and $Q_2^\prime=C_1[x_1^+, {x_1^\prime}^-]\cup \{x_2\}$. Clearly, $|C_1|+|C_2|=|P_1^\prime|+|Q_1^\prime|=|P_2^\prime|+|Q_2^\prime|$.
By Lemma \ref{lemma 4} and the choice of $s$, we have $e(L_2, \cup_{i=3}^sC_i)\le \sum_{i=3}^s|C_i|$ and
\begin{align*}
e(L_2, C_1\cup C_2) & =  e(\{x_1^+, x_2^\prime\}, P_1^\prime)+e(\{x_1^+, x_2^\prime\}, Q_1^\prime)+e(\{x_1^\prime, x_2^-\}, P_2^\prime)+e(\{x_1^\prime, x_2^-\}, Q_2^\prime) \\
 & \mbox{\ $\le  \frac{|P_1^\prime|}{2}+(0+\frac{|Q_1^\prime|}{2})+\frac{|P_2^\prime|}{2}+(\frac{|C_1[x_1^+, {x_1^\prime}^-]|+1}{2}+1)$\ }\\
 & =   |C_1|+|C_2|+1.
\end{align*}
Recall that $e(\{x_2, v\},C_3)\ge \frac{1}{2}|C_3|+1$, it follows from  Lemma \ref{lemma 4} (1) that $G[V(C_3)\cup \{x_2, v\}]$ contains a cycle $C_3^\prime$ covering $S_{C_3}\cup \{x_2\}$. Note that $s$ is minimum, $G[V(C_1\cup C_2\cup H-x_2-v)]$ does not contain a cycle covering $S_{C_1\cup C_2-x_2}$.
This implies that $x_1^+x_2^\prime, x_1^\prime x_2^-\notin E$ and $e(L_2,H)=e(L_2, H-v)+e(L_2, \{v\})\le (|H|-1)+1=|H|$.
Hence, 
\begin{equation*}
\mbox{\ $2(n+1)\le e(L_2, G)\le |H|+(|C_1|+|C_2|+1)+\sum_{i=3}^s|C_i|=2n+1$, \ }
\end{equation*}
 a contradiction.
\end{proof}

\begin{clm}\label{claim 4.3}
Let $i,j\in \{1,\ldots, s\}$ with $i\ne j$. Then $e(X_{C_i}, Y_{C_j})=0$ or $e(Y_{C_i}, X_{C_j})=0$.
\end{clm}
\begin{proof}
Suppose that there exist two disjoint cycles and two distinct edges, say $C_1$, $C_2$ and $x_1y_2, y_1x_2$, such that $x_1y_2\in E(X_{C_1}, Y_{C_2})$ and $y_1x_2\in E(Y_{C_1}, X_{C_2})$. Assume that $x_1\ne y_1^-$ and $x_2\ne y_2^-$. Choose $C_1$, $C_2$, $x_1$, $y_1$, $x_2$ and $y_2$ so that $x_1=y_1^+$ or $x_2=y_2^+$ if possible. 

Let $L=\{y_1^+, x_1^-, y_2^+, x_2^-\}$, $P_1=C_1[y_1^+, y_1]C_2[x_2, x_2^-]$ and $P_2=C_2[y_2^+, y_2]C_1[x_1, x_1^-]$. By the choice of $s$ and Lemma \ref{lemma 4} (2), we have $y_1^+x_2^-, y_2^+x_1^-\notin E$ and $e(L, C_1\cup C_2)\le \frac{|P_1|+|P_2|}{2}=|C_1|+|C_2|$. 
Moreover, by Claim \ref{claim 4.2} we have $y_1^+, y_2^+\in S$ and $d_H(y_1^+)=d_H(y_2^+)=0$. It follows that 
\begin{equation*}
\mbox{\ $e(L, \cup_{j=3}^sC_j)\ge 2(n+1)-2(\frac{|H|}{2}+0)-(|C_1|+|C_2|)=\sum_{j=3}^s|C_j|+2$. \ }
\end{equation*}
Thus there exists a cycle, say $C_3$, such that $e(L, C_3)\ge |C_3|+1$. 
By the symmetry of $\{y_1^+, x_1^-\}$ and $\{y_2^+, x_2^-\}$, we assume that $e(\{y_1^+, x_1^-\}, C_3)\ge \frac{|C_3|}{2}+1$ and then there exists an edge $y_3y_3^+\in E(C_3)$ such that $y_1^+y_3, x_1^-y_3^+\in E$. By the choice of $C_1$, $C_2$, $x_1$, $y_1$, $x_2$ and $y_2$, we have $x_1=y_1^+$ or $x_2=y_2^+$. If $x_1=y_1^+$, then $C_1[y_1^+, y_1]C_3[y_3^+, y_3]y_1^+$ is a cycle which covering $S_{C_1\cup C_3}$, this contrary to the choice of $s$. If $x_2=y_2^+$, then $C_1[y_1^+, x_1^-]C_3[y_3^+, y_3]y_1^+$ and $C_1[x_1, y_1]C_2[x_2, y_2]x_1$ are two disjoint feasible cycles which covering $S_{C_1\cup C_2\cup C_3}$, a contradiction again.
\end{proof}

In the following, we denote the predecessor of $y^-$ by $y^{2-}$, where $y\in V(G)$.

\begin{clm}\label{claim 4.4}
For each $C_i$ with $1\le i\le s$, there exist two disjoint cycles $C_j$, $C_l$ with $j, l\in \{1, \ldots, s\}\setminus \{i\}$ and three vertices $y_i, y_j, y_l$ such that $y_i^-y_j\in E(X_{C_i}, Y_{C_j})$, $y_j^-y_l\in E(X_{C_j}, Y_{C_l})$ and $y_l^-y_i\in E(X_{C_l}, Y_{C_i})$, where $y_i^-, y_j^-, y_l^-$ are the predecessor of $y_i, y_j, y_l$ in $C_i, C_j, C_l$, respectively \emph{(}See Figure 3\emph{)}. 

In particular, if $|C_p|\ge 6$ for some $p\in \{i, j, l\}$, then $y_p^{2-}y_p^+\notin E$, $e(\{y_p^{2-}, y_p^+\},C_p)=\frac{|C_p|}{2}+1$ and $e(\{y_p^{2-}, y_p^+\},C_q)=\frac{|C_q|}{2}$ for each $q\in \{1, \ldots, s\}\setminus \{p\}$.
%
%
%
%
%
\end{clm}
\begin{proof}
By Claim \ref{claim 4.1}, there exists a cycle, say $C_j$, such that $e(C_i, C_j)\ge 1$, assume $y_i^-y_j\in E(C_i, C_j)$. Choose $C_j$ such that $y_i^-\in X_{C_i}$ if possible. 
Consider the path $C_i[y_i, y_i^-]C_j[y_j, y_j^-]$. Since $s$ is minimum, by Lemma \ref{lemma 4} (2), we have $y_iy_j^-\notin E$ and $e(\{y_i, y_j^-\}, C_i\cup C_j)\le \frac{|C_i|+|C_j|}{2}$. 
Moreover, by Claim \ref{claim 4.2}, we have $|\{y_i, y_j^-\}\cap S|=1$ and $e(\{y_i, y_j^-\}, H)\le \frac{|H|}{2}$.
Thus $e(\{y_i, y_j^-\}, \cup_{l\in \{1, \ldots, s\}\setminus\{i,j\}}C_l)\ge (n+1)-\frac{|H|}{2}-\frac{|C_i|+|C_j|}{2}=\sum_{l\in \{1, \ldots, s\}\setminus\{i,j\}}\frac{|C_l|}{2}+1$, which implies that there exist a cycle, say $C_l$, and an edge, say $y_l^-y_l\in E(C_l)$, such that $y_j^-y_l, y_l^-y_i\in E$. Clearly, we can choose $C_j$ such that $y_i^-\in X_{C_i}$ as $|\{y_i^-, y_i\}\cap X_{C_i}|=1$. So $C_j$ and $C_l$ are the cycles we need.


Now assume that $|C_p|\ge 6$ for some $p\in \{i, j, l\}$. If $|S_{C_p}|\ge 3$, then by Lemma \ref{lemma 4} (2) and the choice of $s$ we have $y_p^{2-}y_p^+\notin E$ and $e(\{y_p^{2-}, y_p^{+}\}, C_p)\le \frac{|C_p|-2}{2}+2=\frac{|C_p|}{2}+1$; if $|S_{C_p}|=2$, then by the fact that $\sum_{i=1}^s|C_i|$ is minimum we have $y_p^{2-}y_p^+\notin E$ and $e(\{y_p^{2-}, y_p^{+}\}, C_p)=4$. Recall that $y_p^+\in S$ and $d_H(y_p^+)=0$, by Claim \ref{claim 4.3} we have 
\begin{equation*}
\mbox{\ $n+1\le e(\{y_p^{2-}, y_p^{+}\}, G)\le \max\{\frac{|C_p|}{2}+1, 4\}+\sum_{q\in \{1,\ldots, s\}\setminus \{p\}} \frac{|C_q|}{2}+\frac{|H|}{2}=n+1$.\ }
\end{equation*}
 Thus $e(\{y_p^{2-}, y_p^+\},C_p)=\frac{|C_p|}{2}+1$ and $e(\{y_p^{2-}, y_p^+\},C_q)=\frac{|C_q|}{2}$ for each $q\in \{1, \ldots, s\}\setminus \{p\}$. 
\end{proof}
\begin{center}
\begin{picture}(200,180)\linethickness{0.8pt}
\put(29.6,142.9){\circle{61.6}}
\put(101.6,54.9){\circle{61.6}}
\put(170.1,142.9){\circle{61.6}}
\put(6.5,122.5){\vector(-68,73){1}}
\put(193.2,163.2){\vector(68,-73){1}}
\put(78.5,34.5){\vector(-68,73){1}}
\put(4,136){$C_i$}
\put(185.8,136){$C_j$}
\put(95.8,30.1){$C_l$}
%
\color{black}\polygon*(49,119)(51.5,119)(51.5,121.5)(49,121.5)\polygon(49,119)(51.5,119)(51.5,121.5)(49,121.5)
\color{black}\polygon*(59,146)(61.5,146)(61.5,148.5)(59,148.5)\polygon(59,146)(61.5,146)(61.5,148.5)(59,148.5)
\Line(59,147.6)(139.6,147.6)
\color{black}\polygon*(90,82.5)(92.5,82.5)(92.5,85)(90,85)\polygon(90,82.5)(92.5,82.5)(92.5,85)(90,85)
\color{black}\polygon*(123,74.5)(125.5,74.5)(125.5,77)(123,77)\polygon(123,74.5)(125.5,74.5)(125.5,77)(123,77)
\Line(58.5,135)(91.5,83)
\color{black}\polygon*(139.5,133)(142,133)(142,135.5)(139.5,135.5)\polygon(139.5,133)(142,133)(142,135.5)(139.5,135.5)
\color{black}\polygon*(144,161)(146.5,161)(146.5,163.5)(144,163.5)\polygon(144,161)(146.5,161)(146.5,163.5)(144,163.5)
\Line(140,134)(111.5,83.7)
\put(34.6,158.1){$y_i^{2-}$}
\put(45,146){$y_i^-$}
\put(45,133.1){$y_i$}
\put(35,121){$y_i^+$}
\put(149.5,160){$y_j^+$}
\put(153,121){$y_j^{2-}$}
\put(144,145.5){$y_j$}
\put(91.8,74){$y_l^-$}
\put(114,63.5){$y_l^+$}
\put(106,74){$y_l$}
\put(145,132){$y_j^-$}
\put(76.5,67){$y_l^{2-}$}
\put(59,134){\circle*{3.8}\color{white}\circle*{2.6}}
\put(139.6,147.6){\circle*{3.8}\color{white}\circle*{2.6}}
\put(111.5,83.7){\circle*{3.8}\color{white}\circle*{2.6}}
\put(53.5,162.1){\circle*{3.8}\color{white}\circle*{2.6}}
\put(147.6,121.1){\circle*{3.8}\color{white}\circle*{2.6}}
\put(79,76.3){\circle*{3.8}\color{white}\circle*{2.6}}
\put(185,44){$\in Y$}
\put(185,30){$\in S$}
\put(175,47){\circle*{3.8}\color{white}\circle*{2.6}}
\color{black}\polygon*(174,31.5)(176.5,31.5)(176.5,34)(174,34)\polygon(174,31.5)(176.5,31.5)(176.5,34)(174,34)
\end{picture}

\vspace{-15pt}
\footnotesize{Figure 3. The graph $G[V(C_i\cup C_j\cup C_l)]$}
\end{center}

Now we choose three cycles from $G$ and transform them into two cycles to complete the proof.
Let $|C_1|=\max\{|C_i| : 1\le i\le s\}$. Then by Claim \ref{claim 4.4}, there exist two disjoint cycles, say $C_2$ and $C_3$ and three edges $y_1^-y_2\in E(X_{C_1}, Y_{C_2})$, $y_2^-y_3\in E(X_{C_2}, Y_{C_3})$ and $y_3^-y_1\in E(X_{C_3}, Y_{C_1})$.

%
%
%
%

First suppose that $|C_1|\ge 6$. Then by Claims \ref{claim 4.3}-\ref{claim 4.4}, we have $d_{C_3}(y_1^{2-})=\frac{|C_3|}{2}$, $d_{C_2}(y_1^{+})=\frac{|C_2|}{2}$ and $y_2^{2-}y_3^+\notin E$.
If $y_2^+y_3^{2-}\in E$, then $G[V(C_1\cup C_2\cup C_3)]$ contains two disjoint feasible cycles $y_1^-y_1y_3^-y_3y_2^-y_2y_1^-$ and $C_2[y_2^+, y_2^{2-}]C_1[y_1^+, y_1^{2-}]C_3[y_3^+, y_3^{2-}]y_2^+$ covering $S_{\cup_{i=1}^3C_i}$, this contrary to the choice of $s$. Thus $y_2^+y_3^{2-}\notin E$. So by Claims \ref{claim 4.3}-\ref{claim 4.4} we have $e(\{y_2^{2-}, y_2^+\},C_3)\le \frac{|C_3|}{2}-1$ and $e(\{y_3^{2-}, y_3^+\},C_2)\le \frac{|C_2|}{2}-1$, and thus $|C_2|=|C_3|=4$. Note that by Claim \ref{claim 4.2} we have $y_2^+, y_3^+\in S$ and $e(\{y_2^+, y_3^+\}, H)=0$. It follows from Claim \ref{claim 4.3} that 
\begin{equation*}
\mbox{\ $2(n+1)\le e(\{y_2^{2-}, y_2^+, y_3^{2-}, y_3^+\}, G)\le 2(\frac{|C_1|+\sum_{i=4}^s|C_i|}{2})+(4+\frac{|C_3|}{2}-1+4+\frac{|C_2|}{2}-1)+2(\frac{|H|}{2})=2n+2$.\ }
\end{equation*}
So we have $y_2^+y_3, y_2^{2-}y_1^-\in E$ as $d_{C_3}(y_2^+)= \frac{|C_3|}{2}-1$ and $d_{C_1}(y_2^{2-})= \frac{|C_1|}{2}$, which implies that $G[V(C_1\cup C_2\cup C_3)]$ contains two disjoint feasible cycles $y_1^-y_2^{2-}y_2^-y_3y_2^+y_2y_1^-$ and $C_1[y_1, y_1^{2-}]C_3[y_3^+, y_3^-]y_1$ covering $S_{\cup_{i=1}^3C_i}$, a contradiction.

Thus $|C_i|=4$ for all $i\in \{1,\ldots, s\}$ and $C_p=y_p^{2-}y_p^-y_py_p^+y_p^{2-}$ for each $p\in \{1,2,3\}$, where $S_{C_p}=\{y_p^-, y_p^+\}$. Let $L=\{y_1^{2-},y_1^+, y_2^{2-}, y_2^+, y_3^{2-}, y_3^+\}$. According to Claims \ref{claim 4.2}-\ref{claim 4.3}, we have $e(L,H)\le \frac{3|H|}{2}$ and $e(L, \cup_{i=4}^s C_i)\le \frac{3}{2}\sum_{i=4}^s|C_i|$ and we also have $y_1^{2-}y_2^+, y_2^{2-}y_3^+, y_3^{2-}y_1^+\notin E$.
It follows that 
\begin{equation*}
\mbox{\ $e(\{y_1^{2-}, y_2^+\}, C_3)+e(\{y_2^{2-}, y_3^+\}, C_1)+e(\{y_3^{2-}, y_1^+\}, C_2)\ge 3(n+1)-3(\frac{\sum_{i=4}^s|C_i|+|H|}{2})-6\times 2=9$.\ }
\end{equation*}
 Thus 
\begin{eqnarray*}
 & & e(\{y_1^{2-}, y_2^+\}, \{y_3^{2-}, y_3^+\})+e(\{y_2^{2-}, y_3^+\}, \{y_1^{2-}, y_1^+\})+e(\{y_3^{2-}, y_1^+\}, \{y_2^{2-}, y_2^+\})\ge 5, \\
   & \mbox{ or \ }& e(\{y_1^{2-}, y_2^+\}, \{y_3^{-}, y_3\})+e(\{y_2^{2-}, y_3^+\}, \{y_1^{-}, y_1\})+e(\{y_3^{2-}, y_1^+\}, \{y_2^{-}, y_2\})\ge 5.
\end{eqnarray*}
If the former holds, then $G[V(C_1\cup C_2\cup C_3)]$ contains two disjoint feasible cycles $y_1^-y_2y_2^-y_3y_3^-y_1y_1^-$ and $y_1^{2-}y_3^+y_3^{2-}y_2^+y_2^{2-}y_1^+y_1^{2-}$ covering $S_{\cup_{i=1}^3C_i}$, a contradiction. Therefore the latter must hold. According to the symmetry of $C_1$, $C_2$ and $C_3$, we may assume that $y_1^{2-}y_3^-\notin E$ or $y_1^+y_2\notin E$ if $e(\{y_1^{2-}, y_2^+\}, \{y_3^{-}, y_3\})$ $+$ $e(\{y_2^{2-}, y_3^+\}, \{y_1^{-}, y_1\}) + e(\{y_3^{2-}, y_1^+\}, \{y_2^{-}, y_2\}) = 5$. It follows that $G[V(C_1\cup C_2\cup C_3)]$ contains two disjoint feasible cycles $C_1^\prime$ and $C_2^\prime$, where $C_1^\prime=y_2y_1^+y_1^{2-}y_1^-y_2$, $C_2^\prime=y_1y_3^-y_3^{2-}y_2^-y_2^{2-}y_2^+y_3y_3^+y_1$, or $C_1^\prime=y_3^-y_1^{2-}y_1^+y_1y_3^-$, $C_2^\prime=y_1^-y_2y_2^+y_3y_3^+y_3^{2-}y_2^-y_2^{2-}y_1^-$,
a contradiction.

The proof of Theorem \ref{theorem} is now complete.

\section{Concluding remarks}
\label{concluding remarks}

First we will illustrate Remark (1) in Section \ref{Introduction} with the following four examples (see Figure 4), in which Examples 1 and 2 show that the lower bounds of $|S|$ in Theorems \ref{theorem 2} and \ref{theorem} are necessary, Example 3 shows that the degree condition $\sigma_{1,1}(S)\geq n+1$ is best possible for Theorem \ref{theorem}, and Example 4 shows that $\sigma_{1,1}(S)\ge n$ is necessary for Theorem \ref{theorem 2}.

\noindent\textbf{Example 1.} Let $k$ be a positive even integer. Let $G_1[S,Y]$ be a balanced bipartite graph of order $2n(=2(2k+1))$ satisfies the following: 

$\bullet$ $S=A\cup C\cup \{x\}$ and $Y=B\cup D\cup \{y\}$, where $|A|=|B|=|C|=|D|=k$ and $|S|=|Y|=2k+1$.

$\bullet$ $G_1[A, B]$ and $G_1[C,D]$ are complete bipartite graphs, $E(G_1[A, D])$ is a perfect matching, 

$\bullet$ $N_{G_1}(x)=B\cup\{y\}$ and $N_{G_1}(y)=C\cup\{x\}$.

Clearly, $\delta(G_1)=k+1$ and so $\sigma_{1,1}(S)=2k+2=n+1$. We will show that $G_1$ is not $S$-$k$-cyclable. In fact, if $G_1$ is $S$-$k$-cyclable, then $G_1$ must contains a feasible cycle with 3 vertices of $S$. If $xy$ is in some cycle $C_1$, then $|S_{C_1}|=3$ and $|S_{C_1}\cap A|=|S_{C_1}\cap C|=1$, which implies that $G_1-C_1$ contains at most $2\left \lfloor \frac{k-1}{2} \right \rfloor=k-2$ disjoint feasible cycles as $k$ is even, a contradiction. Thus there is no cycle which contains the edge $xy$. In this case, it is easy to show that $G_1$ is not $S$-$k$-cyclable, too. 

%
%
%
%

\noindent\textbf{Example 2.} Let $p$ be an odd integer with $p\ge 3$ and let $k=\frac{2p^2-3p+1}{2}$. Let $S$ be a subset of $X$ with $|S|=2k$ and $G_2[X,Y]$ a balanced bipartite graph of order $2n$ satisfies the following properties:

$\bullet$ $S=\{x\}\cup S_1\cup \cdots \cup S_{p}$, where $|S_i|=\frac{2k-1}{p}=2p-3$ for $1\leq i\leq p$,

$\bullet$ $Y_1=\{y_1,\ldots, y_p\}$ is a subset of $Y$, $N_{G_2}(x)=Y_1$ and $N_{G_2}(y_i)=S_i\cup\{x\}$ for $1\le i\le p$,

$\bullet$ $G_2[X\setminus\{x\}, Y\setminus Y_1]$ is a complete bipartite graph.

Clearly, $\sigma_{1,1}(S)=\min\{p+(n-1), (n-p+1)+(1+2p-3)\}=n+p-1=n+\frac{\sqrt{16k+1}-1}{4}$. It is easy to see that any cycle which contains $x$ must contains at least three vertices of $S$. Thus $G_2$ is not $S$-$k$-feasible as $|S|=2k$.

\noindent\textbf{Example 3.} Let $G_3[X,Y]$ be a balanced bipartite graph of order $2n$ and $S$ a subset of $X$ such that $d_Y(x)=1$ for some $x\in S$ and $G_3[X\setminus \{x\},Y]$ is a complete bipartite graph. Clearly, $\sigma_{1,1}(S)=1+(n-1)=n$ and $G_3$ is not $S$-$k$-cyclable for all positive integer $k$.

\noindent\textbf{Example 4.} Let $A\subseteq S\subseteq X$ and $B\subseteq Y$ with $|A|=|B|=|S|-(2k-1)$, where $|S|\ge 2k$. Let $G_4[X,Y]$ be a balanced bipartite graph of order $2n$ such that $G_4[X\setminus A,Y]$ is a complete bipartite graph and $E(G_4[A,B])$ is a perfect matching. Clearly, $\sigma_{1,1}(S)=1+(n-|A|)=n+2k-|S|$ and $G_4$ is not $S$-$k$-feasible as $d_{G_4}(x)=1$ for all $x\in A$ and $|\{d_{G_4}(u)\ge 2: u\in S\}|\le |S|-|A|=2k-1$.

\begin{figure}[h]
  \centering
  \includegraphics[width=12cm,height=8cm]{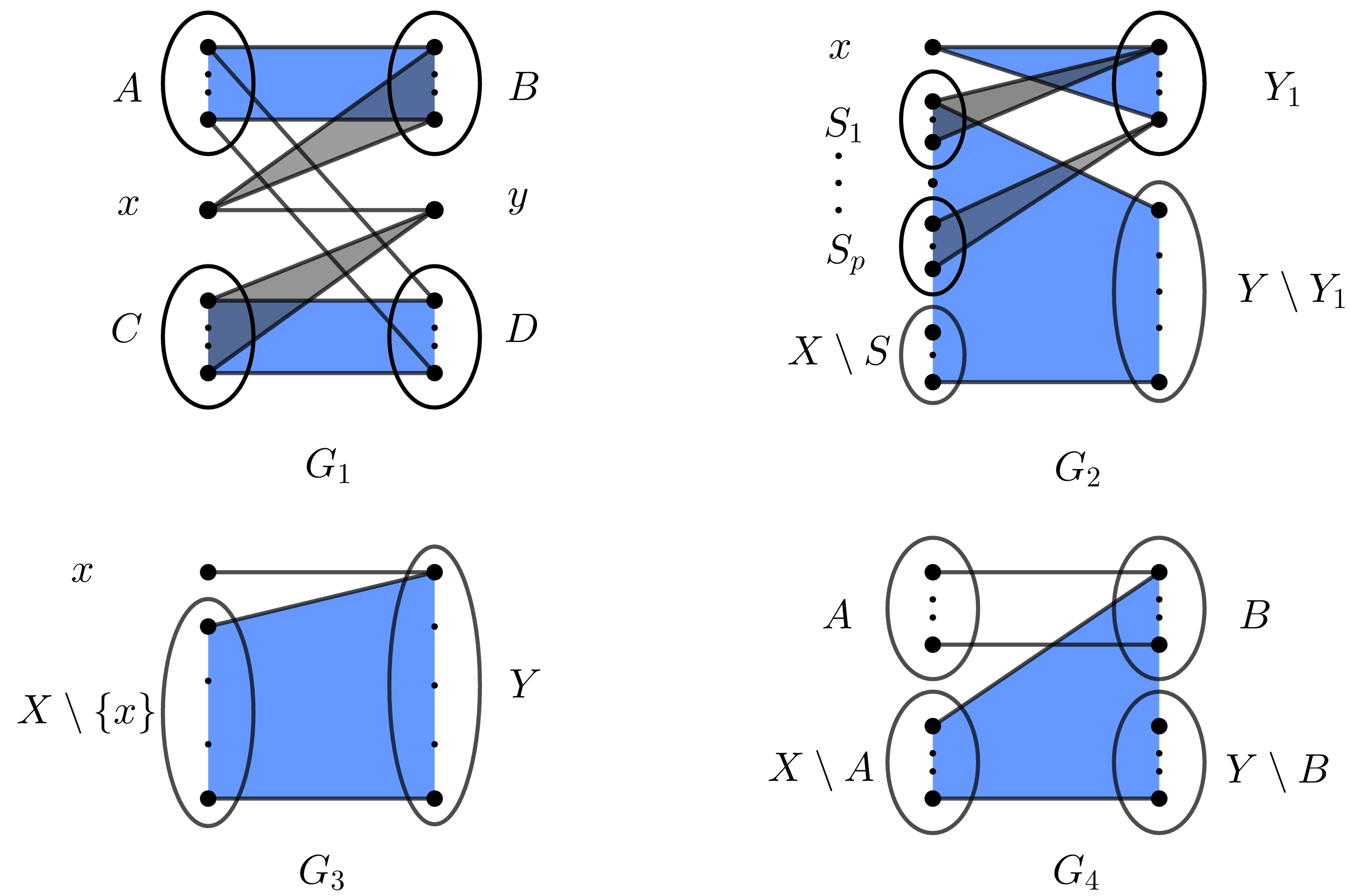}\\
\footnotesize{Figure 4. The graphs $G_1, G_2, G_3$ and $G_4$}
\end{figure}

Based on Examples 1-4 and Theorems \ref{theorem 2}-\ref{theorem}, we propose the following conjectures:


\begin{conjecture}\label{conj1}
Let $k$ be a positive integer. Suppose that $G[X,Y]$ is a balanced bipartite graph of order $2n$, and $S$ is a subset of $X$ with $|S|\geq 2k+1$. If $\sigma_{1,1}(S)\geq n+2k-|S|+1$, then $G$ is $S$-$k$-feasible.
\end{conjecture}

\begin{conjecture}\label{conj2}
Let $k$ be a positive integer. Suppose that $G[X,Y]$ is a balanced bipartite graph of order $2n$, and $S$ is a subset of $X$ with $|S|\geq 2k+2$. If $\sigma_{1,1}(S)\geq n+2$, then for any integer partition $|S|=n_1+\cdots+n_k$ with $n_i\ge 2$ $(1\le i\le k)$, $G$ contains $k$ disjoint cycles $C_1, \ldots, C_k$ such that $|S_{C_i}|=n_i$ for all $1\le i\le k$.
\end{conjecture}

The degree condition $\sigma_{1,1}(S)\geq n+2$ in Conjecture \ref{conj2} is necessary: consider the graph $G_1[S,Y]$ in Example 1 with $2k+1\ge 2\times 2+2$ and suppose that $G_1[S,Y]$ can be partitioned into two cycles $C_1$ and $C_2$ such that $|C_1|=2k$ and $|C_2|=2k+2$, note that the subset $B$ only has $|B|+1$ neighbours, we get $|B_{C_1}|=0$ or $|B_{C_2}|=0$, it is not difficult to see that such cycles $C_1$ and $C_2$ do not exist.  

Conjectures \ref{conj1} and \ref{conj2} has been studied for the case $S=X$. For example, Wang \cite{Wang bipartite} showed that Conjecture \ref{conj1} holds if $\delta(G)\ge k+1$ and $S=X$; Czygrinow, DeBiasio and Kierstead \cite{Czygrinow} showed that Conjecture \ref{conj2} holds if $n$ is sufficiently large in terms of $k$, $\delta(X)+\delta(Y)\ge n+2$ and $S=X$.
For more results on this topic, see \cite{Chiba2018, Gould}.







\begin{thebibliography}{99}
 \bibitem{Abderrezzak}
M. E. K. Abderrezzak, E. Flandrin, D. Amar, Cyclability and pancyclability in bipartite graphs, Discrete Math. 236 (2001) 3--11.

 \bibitem{Bollobas}
B. Bollob\'{a}s, G. Brightwell, Cycles through specified vertices, Combinatorica 13 (1993) 147--155.

 \bibitem{Bondy1976}
J. A. Bondy, V. Chv\'{a}tal, A method in graph theory, Discrete Math. 15 (1976) 111--135.





 \bibitem{Chiba2017}
S. Chiba, T. Yamashita, A note on degree sum conditions for 2-factors with a prescribed number of cycles in bipartite graphs, Discrete Math. 340 (2017) 2871--2877.

 \bibitem{Chiba2018}
S. Chiba, T. Yamashita, Degree conditions for the existence of vertex-disjoint cycles and paths: A survey, Graphs Combin. 34 (2018) 1--83.

 \bibitem{Czygrinow}
A. Czygrinow, L. DeBiasio, H. A. Kierstead, 2-factors of bipartite graphs with asymmetric minimum degrees, SIAM J. Discrete Math. 24 (2010) 486--504.

\bibitem{Diestel}
R. Diestel, Graph Theory, fifth ed., Graduate Texts in Mathematics 173, Springer-Verlag, Berlin Heidelberg, 2017.


  \bibitem{Dirac}
G. A. Dirac, Some theorems on abstract graphs, Proc. Lond. Math. Soc. 2(3) (1952) 69--81.





 \bibitem{Gould}
R. J. Gould, A look at cycles containing specified elements of a graph, Discrete Math. 309 (2009) 6299--6311.

 \bibitem{Jiang}
S. Jiang, J. Yan, Partial degree conditions and cycle coverings in bipartite graphs, Graphs Combin. 33 (2017) 955--967.


 \bibitem{Moon}
J. Moon, L. Moser, On Hamiltonian bipartite graphs, Israel J. Math. 1 (1963) 163--165.


 \bibitem{Shi}
R. Shi, 2-neighborhoods and hamiltonian conditions, J. Graph Theory 16 (1992) 267--271.

\bibitem{Wang bipartite}
 H. Wang, On the maximum number of independent cycles in a bipartite graph, J. Combin. Theory Ser B 67 (1996) 152--164.


 \bibitem{Wang lem}
H. Wang, On 2-factors of a bipartite graph, J. Graph Theory 31 (1999) 101--106.

 \bibitem{Wang Partial}
H. Wang, Partial degree conditions and cycles coverings, J. Graph Theory 78 (2015) 295--304.

\end{thebibliography}
\end{document}